\documentclass[final]{amsart}

\usepackage{amssymb}
\usepackage{amsmath}
\usepackage{mathtools}
\usepackage{microtype}
\usepackage{csquotes}
\usepackage[mathscr]{euscript}

\usepackage{hyperref}
\usepackage{cleveref}
\usepackage{amssymb}
\usepackage{latexsym}
\usepackage{wrapfig}

\usepackage{comment}

\usepackage{graphicx}
\usepackage{todonotes}

\usepackage{tikz}
\usetikzlibrary{matrix,arrows}

\usepackage[all]{xy}
\usepackage{setspace}

\numberwithin{equation}{section}

\theoremstyle{definition}
\newtheorem{definition}[equation]{Definition}
\newtheorem{remark}[equation]{Remark}
\newtheorem{example}[equation]{Example}
\newtheorem{que}[equation]{Question}

\newtheorem*{definition*}{Definition}
\newtheorem*{remark*}{Remark}
\newtheorem*{example*}{Example}
\newtheorem*{que*}{Question}
\newtheorem*{conj*}{Conjecture}
\newtheorem*{constr*}{Construction}

\theoremstyle{plain}
\newtheorem{lem}[equation]{Lemma}
\newtheorem{thm}[equation]{Theorem}
\newtheorem{prop}[equation]{Proposition}
\newtheorem{cor}[equation]{Corollary}

\newtheorem{thma}{Theorem}

\newtheorem{lema}[thma]{Lemma}

\newtheorem{cora}[thma]{Corollary}

\newcommand{\pref}[2]{\hyperref[#2]{#1 \ref*{#2}}}
\hypersetup{urlcolor=blue, citecolor=blue, linkcolor=blue, colorlinks=true} % for screen

\newcommand{\bbF}{{\mathbb{F}}}

\newcommand{\bbR}{{\mathbb{R}}}

\newcommand{\calR}{{\mathcal{R}}}
\newcommand{\calS}{{\mathcal{S}}}

\let\ORGvarepsilon=\varepsilon
\let\varepsilon=\epsilon
\let\epsilon=\ORGvarepsilon

\newcommand{\scal}{{\mathbf{scal}}}

\newcommand{\diff}{{\mathrm{Diff}}}
\newcommand{\ccc}{\widetilde{\pi}_0}

\newcommand{\Spin}{\mathrm{Spin}}
\newcommand{\SO}{\mathrm{SO}}
\newcommand{\ort}{\mathrm{O}}

\newcommand{\im}{\mathrm{im\ }}
\newcommand{\tr}{\mathrm{\textbf{tr }}}
\newcommand{\rk}[1]{\mathrm{{rank}(#1)}}

\newcommand{\congarrow}{\overset{\cong}\longrightarrow}

\newcommand{\embeds}{\hookrightarrow}

\newcommand{\too}{\longrightarrow}

\newcommand{\res}{\mathrm{res}}

\newcommand{\op}{\mathrm{op}}

\newcommand{\rst}{\mathrm{rst}}
\newcommand{\st}{\mathrm{st}}

\newcommand{\id}{\mathrm{id}}

\newcommand{\tor}{\mathrm{tor}}

\newcommand{\conc}{\mathrm{conc}}
\newcommand{\cyl}{\mathrm{cyl}}
\newcommand{\cl}{\mathrm{cl}}

\newcommand{\ie}{\mbox{i.\,e.\,}{}}

\title[$H$-space structures on $\calR^+(M)$]{$H$-Space structures on spaces of metrics of positive scalar curvature}

\author{Georg Frenck}
\thanks{G.F. was supported by the SFB 878 ``Groups, Geometry and Actions'', by the Deutsche Forschungsgemeinschaft (DFG, German Research Foundation) under Germany 's Excellence Strategy – EXC 2044 – 390685587, Mathematics Münster: Dynamics – Geometry - Structure and by the Deutsche Forschungsgemeinschaft (DFG, German Research Foundation) – 281869850 (RTG 2229).}
\address{KIT, Karlsruher Institut für Technologie\\
Englerstra\ss e 2\\
76131 Karlsruhe\\
Bundesrepublik Deutschland}
\email{math@frenck.net}
\email{georg.frenck@kit.edu}
\urladdr{Frenck.net/Math}

\date{\today}


\begin{document}

\begin{abstract}
	We construct and study an $H$-space multiplication on $\calR^+(M)$ for manifolds $M$ which are nullcobordant in their own tangential $2$-type. This is applied to give a rigidity criterion for the action of the diffeomorphism group on $\calR^+(M)$ via pullback. We also compare this to other known multiplicative structures on $\calR^+(M)$.
\end{abstract}

\maketitle

\tableofcontents

\section{Introduction}

\noindent Let $\calR^+(M)$ denote the space of metrics of positive scalar curvature (hereafter: psc-metrics) on a given compact manifold $M$, equipped with the Whitney $C^\infty$-topology. In this paper we will examine multiplicative structures on $\calR^+(M)$. In order to state our results with the least amount of technicalities we confine ourselves to the case of $\Spin$-manifolds. A $\Spin$-manifold $M$ is called $\Spin\times B\pi_1(M)$-nullcobordant if for a classifying map $f\colon M\to B\pi_1(M)$ of the universal cover, the element $[f\colon M\to B\pi_1(M)]$ vanishes in the cobordism group $\Omega^\Spin_d(B\pi_1(M))$. The following is our main theorem (see \pref{Theorem}{thm:main} for the general version).

\begin{thma}\label{thm:mainspin}
	Let $M$ be a $\Spin$-manifold of dimension at least $6$, which is $\Spin\times B\pi_1(M)$-nullcobordant. Then $\calR^+(M)$ is a homotopy-associative, homotopy-com-mutative $H$-space.
\end{thma}

\begin{remark*}
	Note that any $\Spin\times B\pi_1(M)$-nullcobordant manifold of dimension at least $5$ admits a psc-metric as a consequence of the famous Gromov--Lawson--Schoen--Yau surgery theorem (see \cite{gromovlawson_classification} and \cite{schoenyau}).
\end{remark*}

\noindent Our main result applies in particular to high-dimensional spheres, generalizing a result of Walsh \cite{walsh_hspaces}, and products of \emph{arbitrary} $\Spin$-manifolds with $S^n$ for $n\ge2$. The key feature of this $H$-space structure is that the multiplication map is given \enquote{geometrically}. To explain what this means, let us recall the main result from \cite{actionofmcg} (see also \cite{ownthesis}): Let $(M_0,f_0),(M_1,f_1)$ be $(d-1)$-dimensional $\Spin$-manifolds with maps $f_i\colon M_i\to B\pi$ where $\pi\coloneqq\pi_1(M_1)$. We define $\Omega_d^{\Spin,\pi}(M_0,M_1)$ to be the set of equivalence classes of pairs $(W,F)$ of $d$-dimensional $\Spin$-manifolds $W$ together with maps $F\colon W\to B\pi$ such that $\partial W=M_0\amalg M_1$ and $F$ extends $f_0$ and $f_1$. The relation is given as follows: $(W,F)\sim (W',F')$ if there exists a $(d+1)$-dimensional relative $\Spin\times B\pi$-cobordism connecting $(W,F)$ and $(W',F')$, i.e. $\Omega_d^{\Spin,\pi}(M_0,M_1)$ is the set of (relative) cobordism classes of cobordisms from $M_0$ to $M_1$. For spaces $X,Y$ let $[X,Y]$ denote the set of homotopy classes of maps $X\to Y$. In \cite{actionofmcg}\footnote{See also \cite[Section 3.1]{ownthesis}.} we constructed a map
\[\Omega_d^{\Spin,\pi}(M_0,M_1) \too[\calR^+(M_0),\calR^+(M_1)],\]
provided that $d\ge7$ and $f_1$ is a classifying map for the universal cover of $M_1$. We will omit the maps $f,F$. Now let $M$ be a $\Spin$-manifold with fundamental group $\pi$ and let us assume that $M$ is $\Spin\times B\pi$-nullcobordant via $W\colon \emptyset \leadsto M$. This gives a homotopy class of a map $\calS_W\coloneqq\calS(W)\colon\calR^+(\emptyset)\to\calR^+(M)$ and since $\calR^+(\emptyset) = \{g_\emptyset\}$ is a point, we get a base point component of $\calR^+(M)$. Furthermore let $X_W\coloneqq W^\op\amalg W^\op\amalg W\colon M\amalg M\leadsto M$, where $W^\op$ denotes the flipped cobordism. Then the homotopy class of the map
\[\mu_W\coloneqq\calS(X_W)\colon\calR^+(M)\times\calR^+(M)\to\calR^+(M)\]
gives the $H$-space structure in \pref{Theorem}{thm:mainspin} with the neutral element given by $e_W\coloneqq\calS_W(g_\emptyset)$. Since $\mu_W$ only depends on the class of $X_W$ in $\Omega_d^{\Spin,\pi}(M\amalg M, M)$, it is possible to prove \pref{Theorem}{thm:mainspin} by doing computations in this cobordism set. This leads to a form of computation which we call \emph{graphical calculus}. Since the definition $\mu_W$ required the choice of a null-cobordism $W$, it is natural to ask wether $\mu_W$ is independent of this choice. This is answered by the following lemma.

\begin{lema}[{\pref{Lemma}{lem:equivalence-of-hspaces}}]
	Let $M$ and $N$ be $\Spin$-manifolds of dimension at least $6$ with the same fundamental group $\pi$. Let $W\colon \emptyset \leadsto M$ and $W'\colon \emptyset\leadsto N$ be respective $\Spin\times B\pi$-nullcobordisms. Then the map \[\calS(W^\op\amalg W')\colon (\calR^+(M),\mu_W)\to(\calR^+(N),\mu_{W'})\] is an equivalence of $H$-spaces. If $W' = W\amalg B$ for $B$ a closed $\Spin$-manifold with non-vanishing $\alpha$-invariant, then $\calS(W^\op\amalg W')$ does not fix any path component and in particular is not homotopic to the identity.
\end{lema}

\noindent We also show that the components of invertible elements are independent of the nullcobordism $W$ (see \pref{Proposition}{prop:units}). If furthermore $N$ is a (not necessarily nullcobordant) $\Spin$-manifold with the same fundamental group $\pi$, then we define a map
\[\rho_W\coloneqq \calS(N\times[0,1]\amalg W^\op )\colon \calR^+(M) \times\calR^+(N)\too\calR^+(N)\] 
which gives an action of $\calR^+(M)$ on $\calR^+(N)$ in the homotopy category (see \pref{Proposition}{prop:action}). Using graphical calculus we obtain a triviality criterion for the action of the oriented diffeomorphism group $\diff(N)$ on $\calR^+(N)$ in the case $\pi=1$. Note that for an orientation preserving diffeomorphism $f\colon N\to N$ of a simply connected $\Spin$-manifold $N$ there exist $2$ $\Spin$-structures on the mapping torus $T_f\coloneqq N\times[0,1]/ (f(x),1)\sim (x,0)$.

\begin{thma}[{\pref{Theorem}{thm:diffeo-action}}]\label{thm:diffeo-action-spin}
	Let $N,M$ be simply connected $\Spin$-manifolds of dimension at least $6$, let $W\colon\emptyset\leadsto M$ be a $\Spin$-cobordism and let $f\colon N\to N$ be an orientation preserving diffeomorphism. Then $f^*\colon\calR^+(N)\to\calR^+(N)$ is homotopic to the identity if there exists a $\Spin$-structure on $T_f$ such that $e_W$ is isotopic to $\calS(M\times[0,1]\amalg T_f)(e_W)$. If $N$ is $\Spin$-nullcobordant equivalence holds.
\end{thma}

\begin{remark*}
	Considering the special case that $M=N$ we get that $f^*$ is homotopic to the identity if and only if $f^*e_W\sim e_W$. This extends \cite[Proposition D]{actionofmcg}.
\end{remark*}

\noindent In the final \pref{Section}{sec:comparison} we compare $\mu_W$ to other multiplicative structures on $\calR^+(M)$.  We show that Walsh's multiplication from \cite{walsh_hspaces} agrees with $\mu_D$ for the disk $D\colon \emptyset\leadsto S^{d-1}$ provided that $d\ge7$. We then examine the multiplicative structure on concordance classes introduced by Stolz in \cite{stolz_concordance} and further studied in \cite{weinbergeryu} and \cite{xyz}. We show that this is induced by a map of spaces and if the manifold is $\Spin\times B\pi$-nullcobordant it is induced by $\mu_W$. Finally we examine the $H$-multiplication $\mu_\cyl$ given by concatenation of metrics on cylinders. It is shown in \cite{erw_psc3} that for a certain class of manifolds this yields an infinite loop space structure on the subspace of so-called stable metrics. In the special case of the cylinder over a sphere we show that gluing in the torpedo metric on both sides yields an equivalence of $H$-spaces 
\[(\calR^+(S^{d-2}\times[0,1])_{g_\circ,g_\circ},\mu_\cyl) \too (\calR^+(S^{d-1}),\mu_D).\]
As a corollary we get the following.
\begin{cora}\label{cor:compare}
	The underlying $H$-space structures of the $(d-1)$-fold loop space structure from \cite{walsh_hspaces} and the infinite loop space structure from \cite{erw_psc3} on $\calR^+(S^{d-1})^\st$ agree for $d\ge 7$.
\end{cora}

\noindent\textbf{Acknowledgements:} This paper grew out of a part of the author's Münster PhD-thesis and it is my great pleasure to thank my advisor Johannes Ebert for his guidance and many enlightening discussions. I would also like to thank the anonymous referee for his careful reading and very helpful remarks and suggestions which particularly improved the final section. 

\section{Tangential structures and the surgery map}\label{sec:surgerymap}
\noindent For $d\ge0$ let $B\ort(d+1)$ be the classifying space of the $(d+1)$-dimensional orthogonal group and let $U_{d+1}$ be the universal vector bundle over $B\ort(d+1)$. Let $\theta\colon B\to B\ort(d+1)$ be a fibration. We call $\theta$ a \emph{tangential structure}. 
	\begin{definition}
		A \emph{$\theta$-structure} on a real $\rk {d+1}$-vector bundle $V\to X$ is a bundle map $\hat l\colon V\to \theta^*U_{d+1}$. A \emph{$\theta$-structure on a manifold} $W^{{d+1}}$ is a $\theta$-structure on $TW$ and a \emph{$\theta$-manifold} is a pair $(W,\hat l)$ consisting of a manifold $W$ and a $\theta$-structure $\hat l$ on $W$. For $0\le k\le d$ a \emph{stabilized $\theta$-structure} on $M^k$ is a $\theta$-structure on $TM\oplus\underline\bbR^{{d+1}-k}$. 
	\end{definition}
	
	\begin{definition}
		Let $\theta\colon B\to B\ort(d+1)$ be a tangential structure. We call $\theta$ the \emph{(stabilized) tangential $2$-type of a $(d-1)$-dimensional manifold $M$} if the map $\theta$ is $2$-coconnected and there exists a (stabilized) $\theta$-structure $\hat l$ on $M$ such that the underlying map $l\colon M\to B$ is $2$-connected. 
	\end{definition} 
	
	\begin{example}[{\cite[Example 3.3]{actionofmcg}}, {\cite[Example 1.1.6]{ownthesis}}]\label{ex:tangential-structures}\leavevmode
		\begin{enumerate}
			\item The (stabilized) tangential $2$-type of a connected spin manifold $M$ of dimension at least $3$ is $B\Spin(d+1)\times B\pi_1(M)$.
			\item The (stabilized) tangential $2$-type of a simply connected, non-spinnable manifold $M$ of dimension at least $3$ is $B\SO(d+1)$.
		\end{enumerate}
	\end{example}

	\begin{definition}\label{def:structured-bordism-set}		
		Let $M_0^{d-1},M_1^{d-1}$ be closed manifolds with (stabilized) $\theta$-structures $\hat l_0,\hat l_1$. We define the \emph{cobordism set of manifolds with $\theta$-structure and fixed boundary}  by 
		\[\Omega^\theta_d\bigl((M_0,\hat l_0),(M_1,\hat l_1)\bigr):=\bigr\{(W,\hat\ell)\bigl\}/\sim.\]
		Here, $W$ is a $d$-manifold with boundary $\partial W = M_0\amalg M_1$ and $\hat\ell$ is a stabilized $\theta$-structure on $W$ such that $(-1)^i\hat l_i = \hat\ell|_{M_i}$. We call $M_0$ the \emph{incoming boundary} and $M_1$ the \emph{outgoing boundary}  (see \pref{Figure}{fig:bordism-set}). 
		\begin{figure}[ht]
		\centering
		\includegraphics[width=0.45\textwidth]{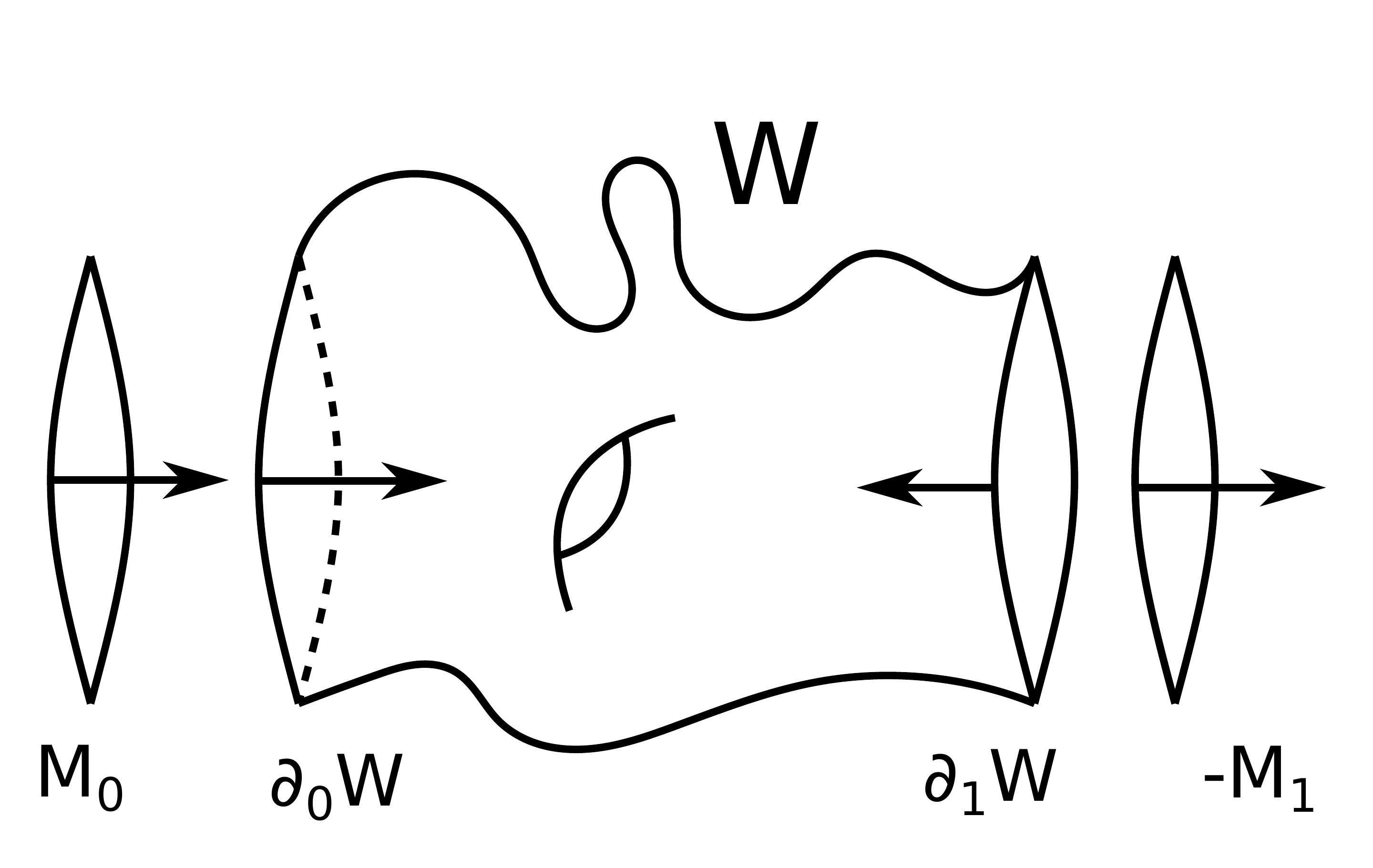}
		\caption{A representative of an element in $\Omega_{d}^\Spin(M_0,M_1)$.}\label{fig:bordism-set}
		\end{figure}\\
		The equivalence relation is given by the relative cobordism relation: We say that $(W,\ell)$ and $(W',\ell')$ are $\theta$-cobordant if there exists a $(d+1)$-dimensional $\theta$-manifold $(X,\ell_X)$ with corners such that there exists a partition of 
		\[\partial X=M_0\times I \cup W \cup M_1\times I\cup W'\]
		such that the $\theta$-structures fit together (see \pref{Figure}{fig:bordism-set-relation}).
		\begin{figure}[ht]
		\centering
		\includegraphics[width=0.7\textwidth]{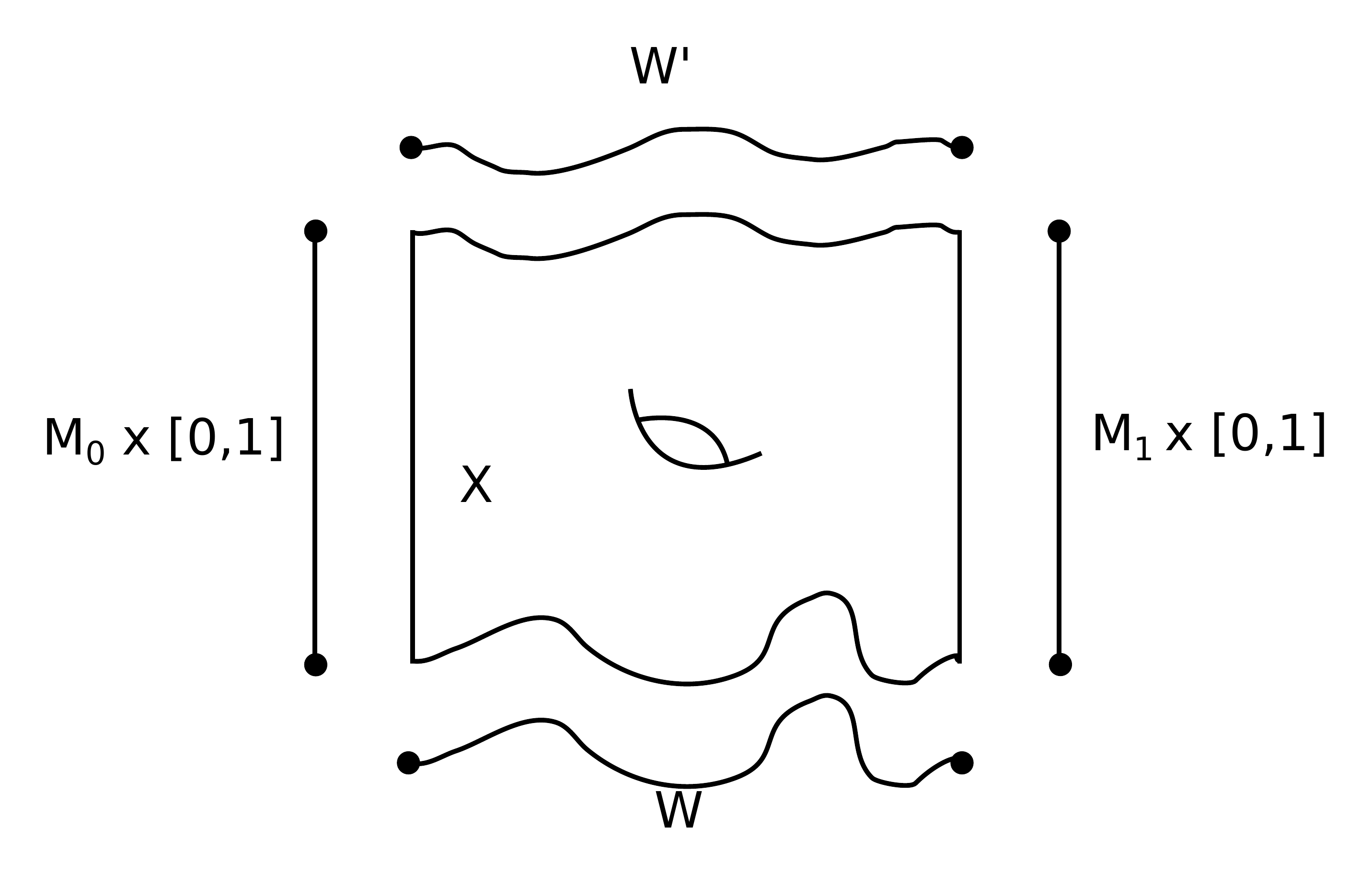}
		\caption{The cobordism relation.}\label{fig:bordism-set-relation}
		\end{figure}
	\end{definition}
	
	\noindent The main theorem of \cite{actionofmcg} is the following:
	\begin{thm}[{\cite[Theorem 3.6]{actionofmcg}}, see also {\cite[Theorem 3.3.1]{ownthesis}}]
		Let $d\ge7$ and let $\theta$ be a $2$-coconnected tangential structure. Let $(M_0,\hat l_0)$ and $(M_1,\hat l_1)$ be $(d-1)$-dimensional $\theta$-manifolds such that the underlying map $l_1\colon M_1\to B$ is $2$-connected. Then there is a map 
		\[\calS\colon\Omega_d^\theta((M_0,\hat l_0),(M_1,\hat l_1))\too[\calR^+(M_0),\calR^+(M_1)]\]
		such that $\calS(M_1\times[0,1]) = \id$ and $\calS$ is compatible with compositions, \ie $\calS(W\cup W') = \calS(W')\circ\calS(W)$.
	\end{thm}
	\noindent $\calS$ is called the \emph{surgery map} and we will sometimes write $\calS(W) = \calS_W$. Note that if $B$ is not connected, say $B=B'\amalg B''$, and $M_i = M_i'\amalg M_i''$ for $i=0,1$ and we have 
	\[\Omega_d^\theta(M_0,M_1) = \Omega_d^{\theta'}(M_0',M_1')\times \Omega_d^{\theta''}(M_0'',M_1'').\]  

	\noindent The following proposition is one of the key features of the cobordism relation.
	\begin{prop}[{\cite[Proposition 3.25]{actionofmcg}}, see also {\cite[Proposition 1.3.3]{ownthesis}})]\label{prop:double-is-nullbordant}
		Let $W^d\colon M_0 \leadsto M_1$ be a $\theta$-cobordism. Then there exists a $\theta$-structure on $W^\op\colon M_1\leadsto M_0$ such that $W\cup W^\op\sim M_0\times [0,1]\ \mathrm{relative\ to}\ M_0\times\{0,1\}$. In particular, if $W\colon \emptyset\leadsto M$ is a $\theta$-nullcobordism, the double $dW\coloneqq W\cup W^\op$ is $\theta$-nullcobordant and $W^\op\amalg W$ is $\theta$-cobordant to the cylinder $M\times[0,1]$.
	\end{prop}

\noindent Let us close this section by recalling the definition $H$-spaces. From now on the symbol \enquote{$=$} will denote equality in the homotopy category of spaces, i.e. $f=f'$ means $f$ and $f'$ are homotopic. Let us start by recalling the notion of an $H$-space.
\begin{definition}
	An \emph{$H$-space} is a triple $(X,\mu,e)$ where $X$ is a space, $e\in X$ and $\mu\colon X\times X\to X$ is a homotopy class of a map, such that $\mu(e,\_)=\mu(\_,e)=\id$. An $H$-space is called \emph{homotopy-commutative} if $\mu\circ \tau=\mu$, for $\tau\colon X\times X\to X\times X$ the switch map and it is called \emph{homotopy-associative} if $\mu\circ (\mu,\id)=\mu\circ(\id,\mu)$. An equivalence of $H$-spaces $(X,\mu,e)$ and $(X', \mu', e')$ is a (homotopy class of a) homotopy equivalence $\varphi\colon X\to X'$ such that $\mu\circ (\varphi,\varphi)=\varphi\circ\mu$ and $\varphi(e) \sim e'$. 
\end{definition}

\begin{remark}
	Usually the definition of an $H$-space involves the choice of an actual map $X\times X\to X$. The definition given here is more in spirit of an $H$-space being a unital magma object in the homotopy category of spaces. Furthermore, since the neutral element of an $H$-space is only well-defined and unique up to homotopy it suffices to specify the component of $e$. 
\end{remark}

\begin{definition}
	Let $Y$ be a space and let $X=(X,\mu,e)$ be an $H$-space. An \emph{action of $X$ on $Y$ in the homotopy category} is a homotopy class of a map
	\[\rho\colon X\times Y\to Y,\]
	such that $\rho(e,\_) = \id_Y$ and $\rho(\mu,\id) = \rho(\id,\rho)$.
\end{definition}

\section{Graphical calculus}\label{sec:graphicalcalculus}

\noindent Let $d\ge7$, let $M^{d-1}$ be a manifold and let $\theta$ be its tangential $2$-type. Let $\hat l$ be a $2$-connected $\theta$-structure and let $W\colon \emptyset\leadsto M$ be a $\theta$-nullcobordism of $(M, \hat l)$.We get a map $\calS(W)\colon\calR^+(\emptyset) = \{g_\emptyset\}\to\calR^+(M)$ which gives a base-point component $e_W$ of $\calR^+(M)$. Furthermore, let $X_W\coloneqq W^\op\amalg W^\op\amalg W\colon M\amalg M\leadsto M$ (see \pref{Figure}{fig:definition-of-h-space}). We define 
\[\mu_W\coloneqq S(X_W)\colon\calR^+(M)\times \calR^+(M)\to\calR^+(M)\]
\vspace{-1em}
 \begin{figure}[ht]
	\centering
	\includegraphics[width=0.25\textwidth]{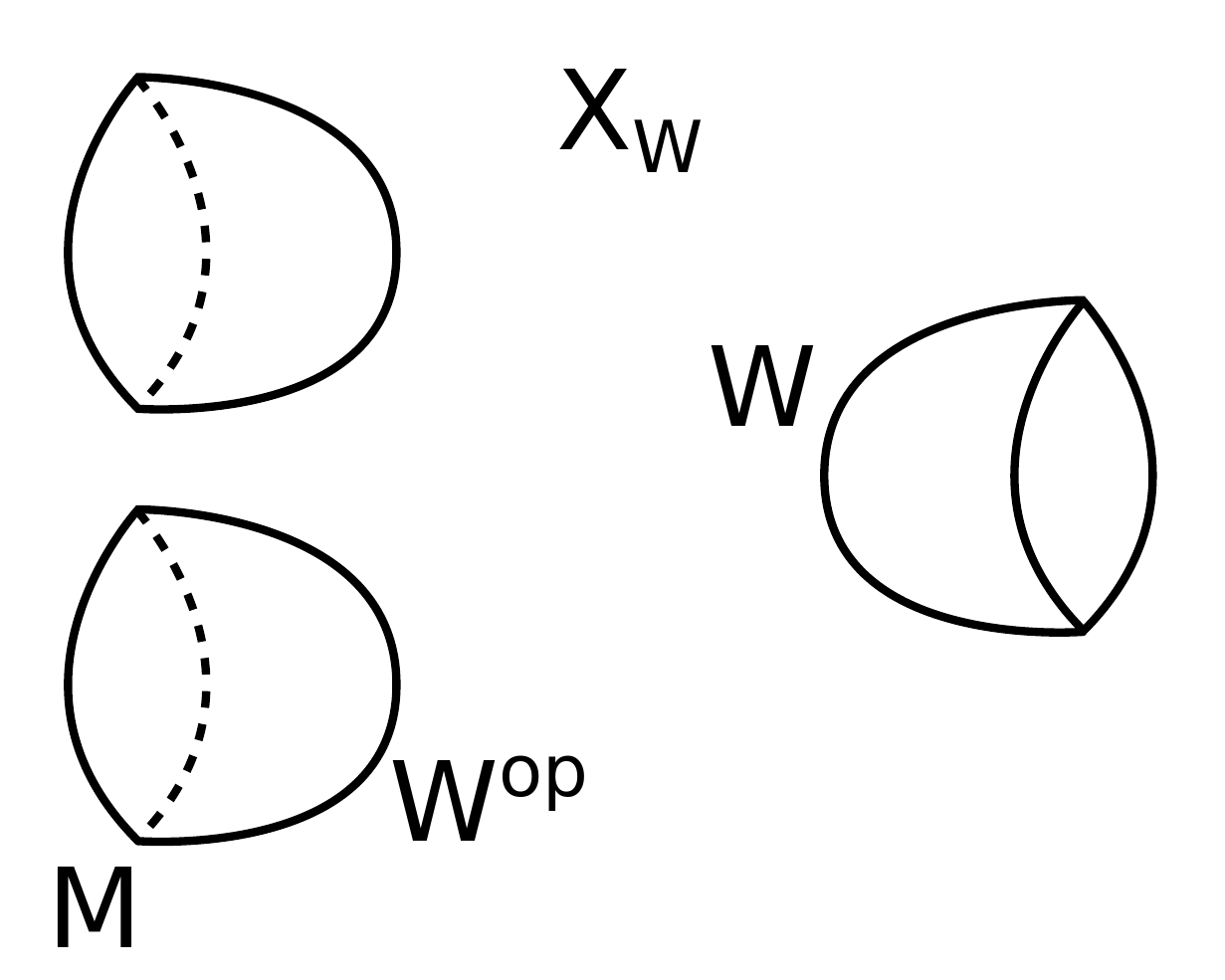}
	\caption{The $\theta$-cobordism $X_W\colon M\amalg M\leadsto M$.}\label{fig:definition-of-h-space}
\end{figure}

\begin{thm}\label{thm:main}
	$(\calR^+(M), \mu_W, e_W)$ is a homotopy-commutative, homotopy-associa-tive $H$-space.
\end{thm}
 
\begin{proof}
	First we show that $e_W$ really is the neutral element. We need to show that $\mu_W\circ (\id,\calS(W))$ is homotopic to the identity. Now $(\id,\calS(W)) = \calS_{(M\times I)\amalg W}$ and so
	\[\mu_W\circ(\id,\calS(W))=\calS_{X_W}\circ \calS_{(M\times I)\amalg W} = \calS_{(M\times I)\cup W^{\op}\amalg dW \amalg W} = \calS_{(M\times I)\cup (M\times I)} \sim \id\]
	as the double of $W$ is nullcobordant by \pref{Proposition}{prop:double-is-nullbordant}. This computation relies on the cobordism relation and is depicted in \pref{Figure}{fig:unital}.
	\begin{figure}[ht]
	\centering
	\includegraphics[width=0.5\textwidth]{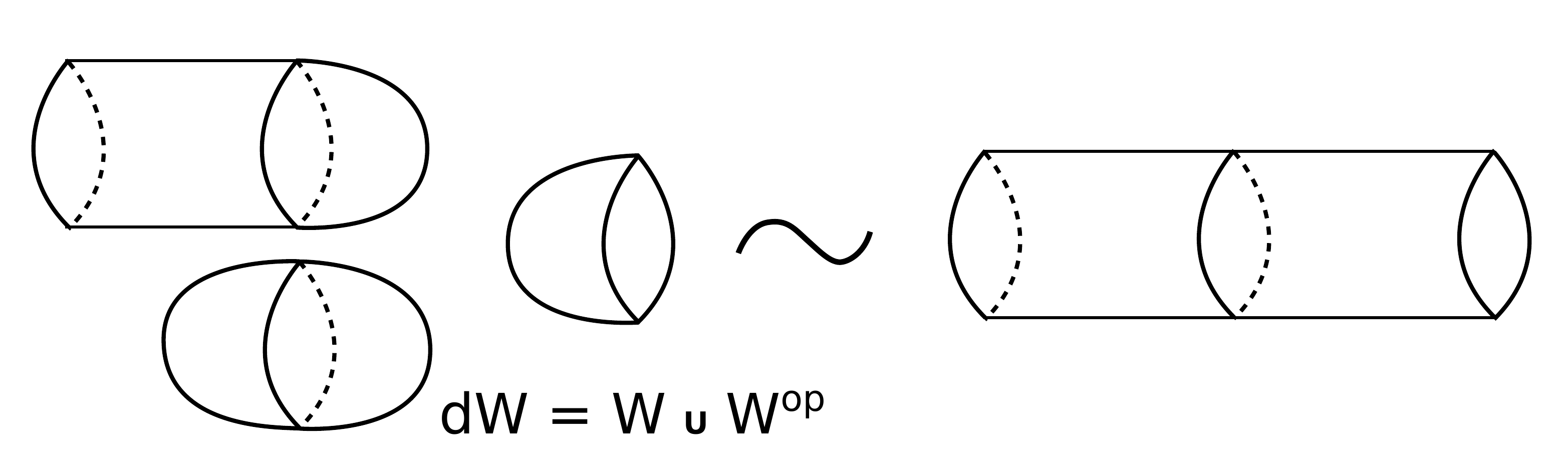}
	\caption{$\mu_W\circ (\id,\calS(W))=\id$}\label{fig:unital}
	\end{figure}\\
	For commutativity, the composition $\mu_W\circ \tau$, where $\tau$ is the map switching the factors, has to be homotopic to $\mu_W$. The map $\tau$ however is given by the surgery map $\calS$ for the cobordism in \pref{Figure}{fig:commutative} and the composition of this cobordism with $X_W$ is cobordant to $X_W$ relative to the boundary. 
	\begin{figure}[ht]
	\centering
	\includegraphics[width=0.6\textwidth]{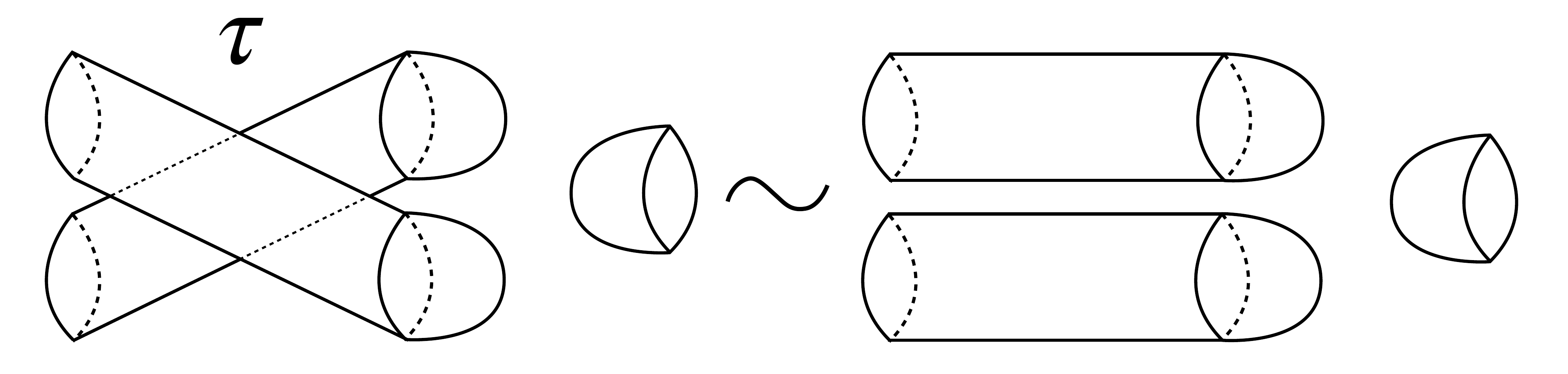}
	\caption{$\mu_W\circ\tau = \mu_W$.}\label{fig:commutative}
	\end{figure}\\
	For associativity we need to show that $\mu\circ (\mu,\id)=\mu\circ(\id,\mu)$. Again, all maps are given by surgery maps and the proof is finished by \pref{Figure}{fig:associative}.
	\begin{figure}[ht]
	\centering
	\includegraphics[width=0.6\textwidth]{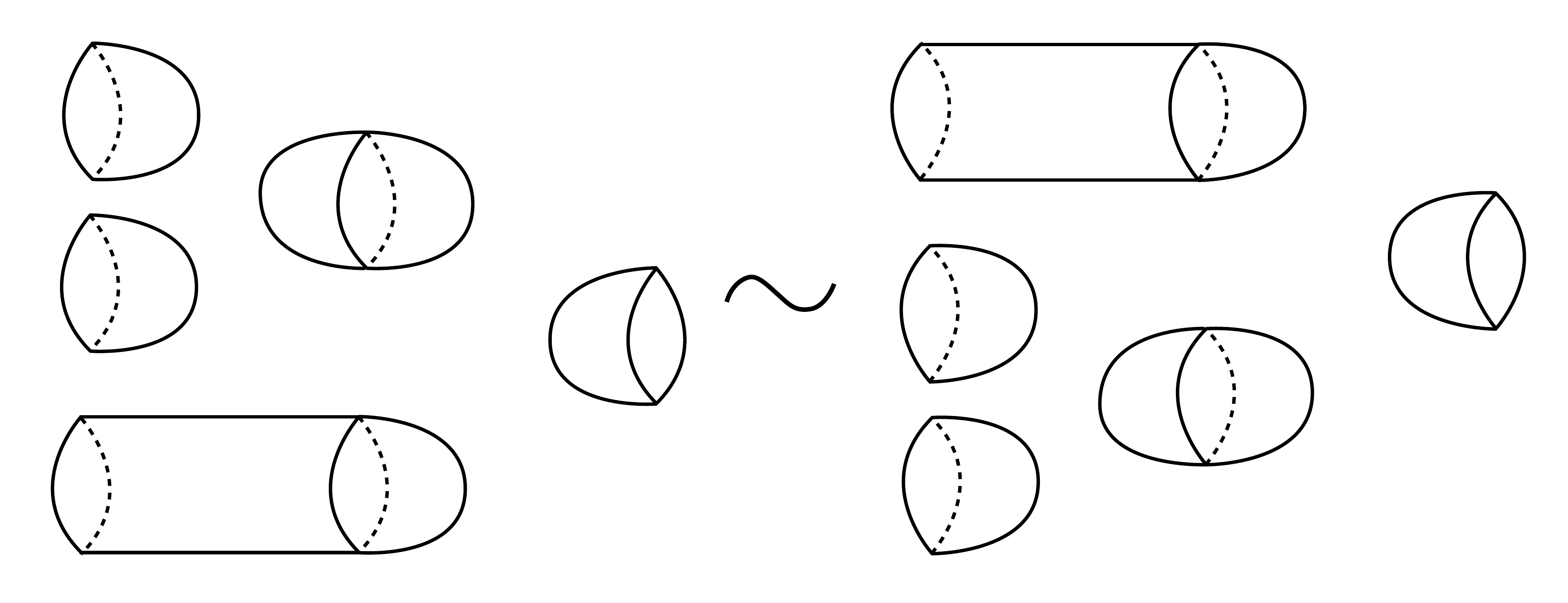}
	\caption{$\mu\circ (\mu,\id)=\mu\circ(\id,\mu)$.}\label{fig:associative}
	\end{figure}\\\ \qedhere
\end{proof}

\begin{cor}
	The set $\pi_0(\calR^+(M))$ carries the structure of an abelian monoid induced by $\mu_W$, $\pi_1(\calR^+(M),e_W)$ is an abelian group and $H^*(\calR^+(M);\bbF)$ is a graded Hopf algebra for any field $\bbF$.
\end{cor}  
\newpage

\begin{remark}
A word of warning is appropriate here: Using pictures to do computations can be dangerous as illustrated by the following example: consider the cobordism $X\coloneqq W^{\op}\amalg W^{\op}\amalg W\amalg W\colon M\amalg M\leadsto M\amalg M$ (see \pref{Figure}{fig:warning}). 
	\begin{figure}[ht]
	\centering
	\includegraphics[width=0.2\textwidth]{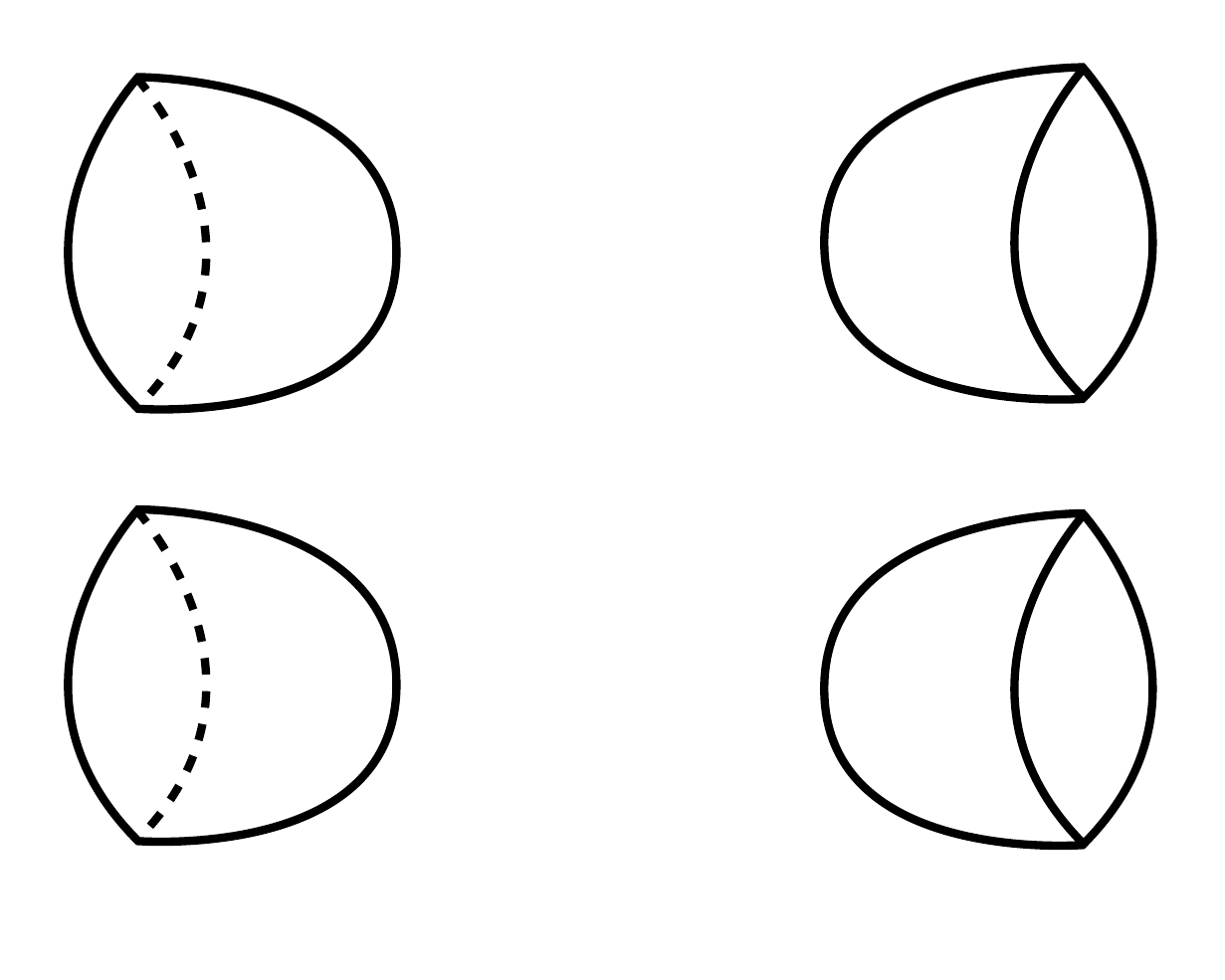}
	\caption{$X\coloneqq W^{\op}\amalg W^{\op}\amalg W\amalg W\colon M\amalg M\leadsto M\amalg M$}\label{fig:warning}
	\end{figure}
	
	\noindent We then have two decompositions $(W^{\op}\amalg W) \amalg (W^{\op}\amalg W)=X=X_W\amalg W$ of $X$. One might be tempted to think that $(\mu_W,e) =  \calS_{X_W\amalg W} = \calS_{(W^{\op}\amalg W) \amalg (W^{\op}\amalg W)} = (\id,\id)$ implying that $\calR^+(M)$ is contractible. This computation is wrong, as one needs to consider the tangential $2$-type of the outgoing boundary which is not connected in this case. Hence the corresponding tangential $2$-type $\theta\colon B\to BO(d+1)$ is not connected (in the sense that $B$ is not connected) and a $\theta$-structure on $W$ is a map into a disconnected space that respects the given one on the outgoing boundary. Therefore one has to specify which component of $W$ is mapped to which component of $B$. In particular, the components of the incoming boundary are already coupled with components of the outgoing boundary. The manifolds $(W^{\op}\amalg W) \amalg (W^{\op}\amalg W)$ and $X_W\amalg W$ are different as $\theta$-manifolds, even though the underlying manifolds are equal.
	
	However, when the outgoing boundary is connected so is the corresponding tangential $2$-type and one does not need to be as careful. This is the case in the computations in the proof of \pref{Theorem}{thm:main} and will be in every computation in this section.
\end{remark}

\begin{example}\label{ex:sphere-h-space}
	By the definition of $\calS$ we get $e_D = g_\circ^{d-1}$ for $D = D^{d}\colon\emptyset\leadsto S^{d-1}$.
\end{example}

\noindent The next lemma explains the dependence of $\mu_W$ on $W$ and on $M$.

\begin{lem}\label{lem:equivalence-of-hspaces}
	Let $W\colon\emptyset\leadsto M$ and $V\colon \emptyset\leadsto N$ be to $\theta$-nullcobordisms. Then 
	\[\varphi\coloneqq \calS(W^\op\amalg V)\colon(\calR^+(M),\mu_W,e_W)\too(\calR^+(N),\mu_V,e_V)\]
	is an equivalence of $H$-spaces. If $M=N$ is simply connected and $\Spin$ and $V=W\amalg B$ for a closed $\Spin$-manifold $B$ with non-vanishing $\alpha$-invariant (cf. \cite[p. 92]{lawsonmichelsohn}), then $\varphi$ does not fix any path component and in particular is not homotopic to the identity.
\end{lem}
\begin{proof}	
	An inverse is given by $\calS_{V^{\op}\amalg W}$, so $\varphi$ is a homotopy equivalence. We have $\varphi\circ \mu_W = \mu_V\circ(\varphi,\varphi)$ because of \pref{Figure}{fig:homomorphism} and $e_V = \varphi(e_W)$ because of \pref{Figure}{fig:unit-computation}.
	\begin{figure}[ht]
	\centering
	\includegraphics[width=0.6\textwidth]{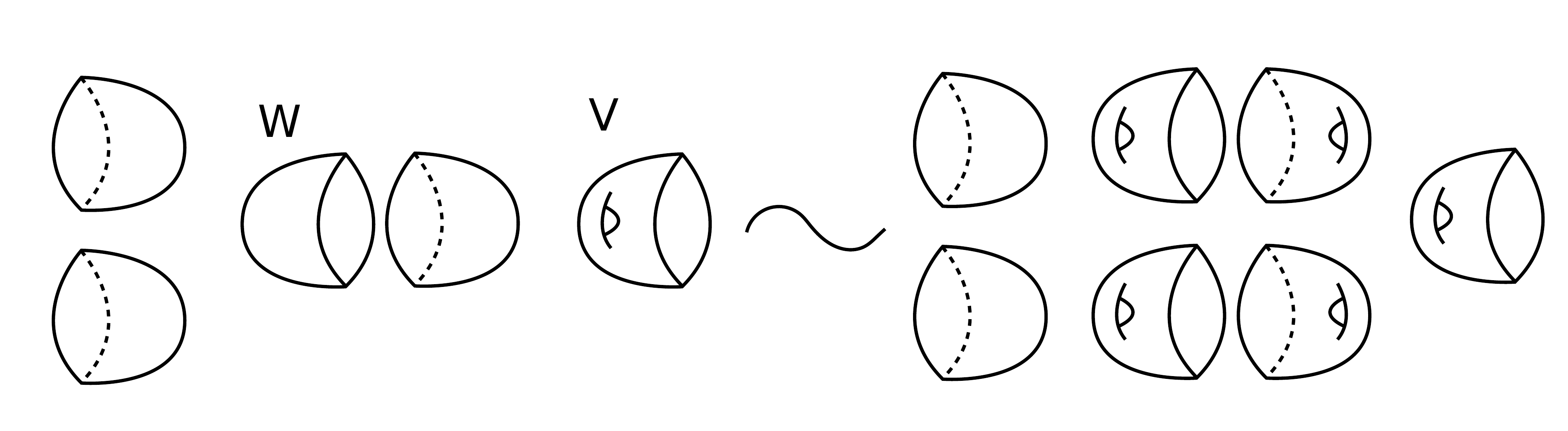}
	\caption{$\varphi\circ \mu_W = \mu_V\circ(\varphi,\varphi)$.}\label{fig:homomorphism}
	\end{figure}
	\begin{figure}[ht]
	\centering
	\includegraphics[width=0.3\textwidth]{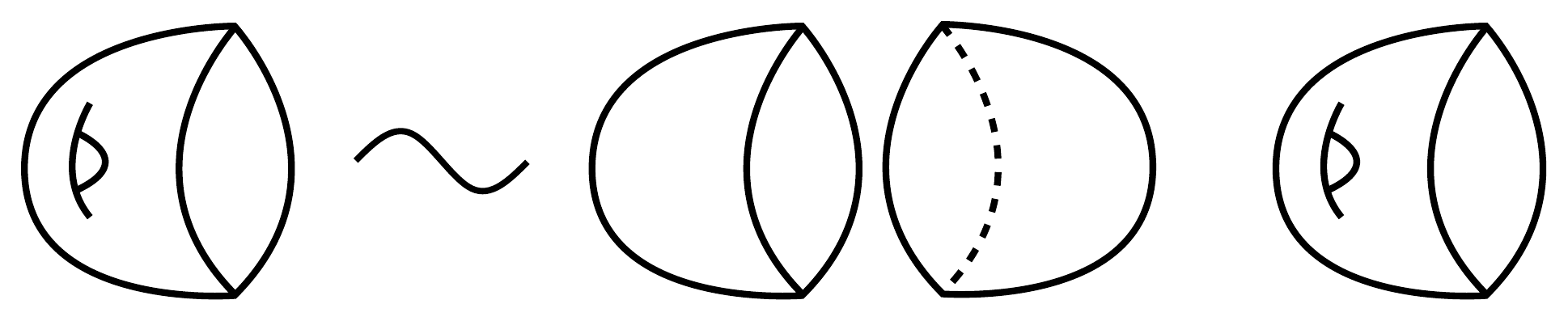}
	\caption{$e_V = \varphi(e_W)$.}\label{fig:unit-computation}
	\end{figure}
	
	\noindent The final part follows from \pref{Proposition}{prop:double-is-nullbordant} and \cite[Proposition 3.35]{actionofmcg}\footnote{see also \cite[Proposition 4.4.3]{ownthesis}}.
\end{proof} 

\noindent Even though $\mu_W$ and $\mu_V$ might be different maps, the path components of invertible elements are the same. Let $G_W$ denote the components of invertible elements with respect to $\mu_W$.

\begin{prop}\label{prop:units}
	Let $V,W\colon\emptyset\leadsto M$ be two $\theta$-nullcobordisms. Then $G_W=\varphi(G_W)=G_V$.
\end{prop}

\noindent This follows from the following, more general lemma.
 
\begin{lem}\label{lem:units}
	Let $U\colon M\leadsto M$ be a $\theta$-cobordism. Then
	\[\mu_W\circ (\calS(U),\id) = \mu_W\circ(\id,\calS(U))=\calS(U)\circ\mu_W\]
	and in particular $\calS(U)(G_W)= G_W$.
\end{lem}

\begin{proof}
	Since $W^{\op}\amalg W$ is cobordant to $M\times I$, the formula follows from \pref{Figure}{fig:units}. For the second part let $g,g'\in G_W$ such that $\mu_W(g,g')\sim e_W$ and let $\calS(U^\op)(g')\sim g''\in \pi_0(\calR^+(M))$. Then $\mu_W(\calS(U)(g),g'') \sim \mu_W(g,\calS(U)(g'')) \sim \mu_W(g,g') \sim e$ and so we have $\calS(U)(G_W)\subset G_W$. The other inclusion follows by the same argument for $U^\op$.\qedhere
\end{proof}

\noindent Now, let $M$ be as before and let $N$ be a manifold with the same tangential $2$-type but not necessarily $\theta$-nullcobordant. We get a $\theta$-cobordism $Y_W\coloneqq W^\op\amalg N\times[0,1]\colon M\amalg N \leadsto N$ (see \pref{Figure}{fig:action}) and a surgery map
	\[\rho_W\coloneqq \calS(Y_W)\colon\calR^+(M)\times\calR^+(N)\too\calR^+(N).\]
	\vspace{-1.7em}
	\begin{figure}[ht]
	\centering
	\includegraphics[width=0.24\textwidth]{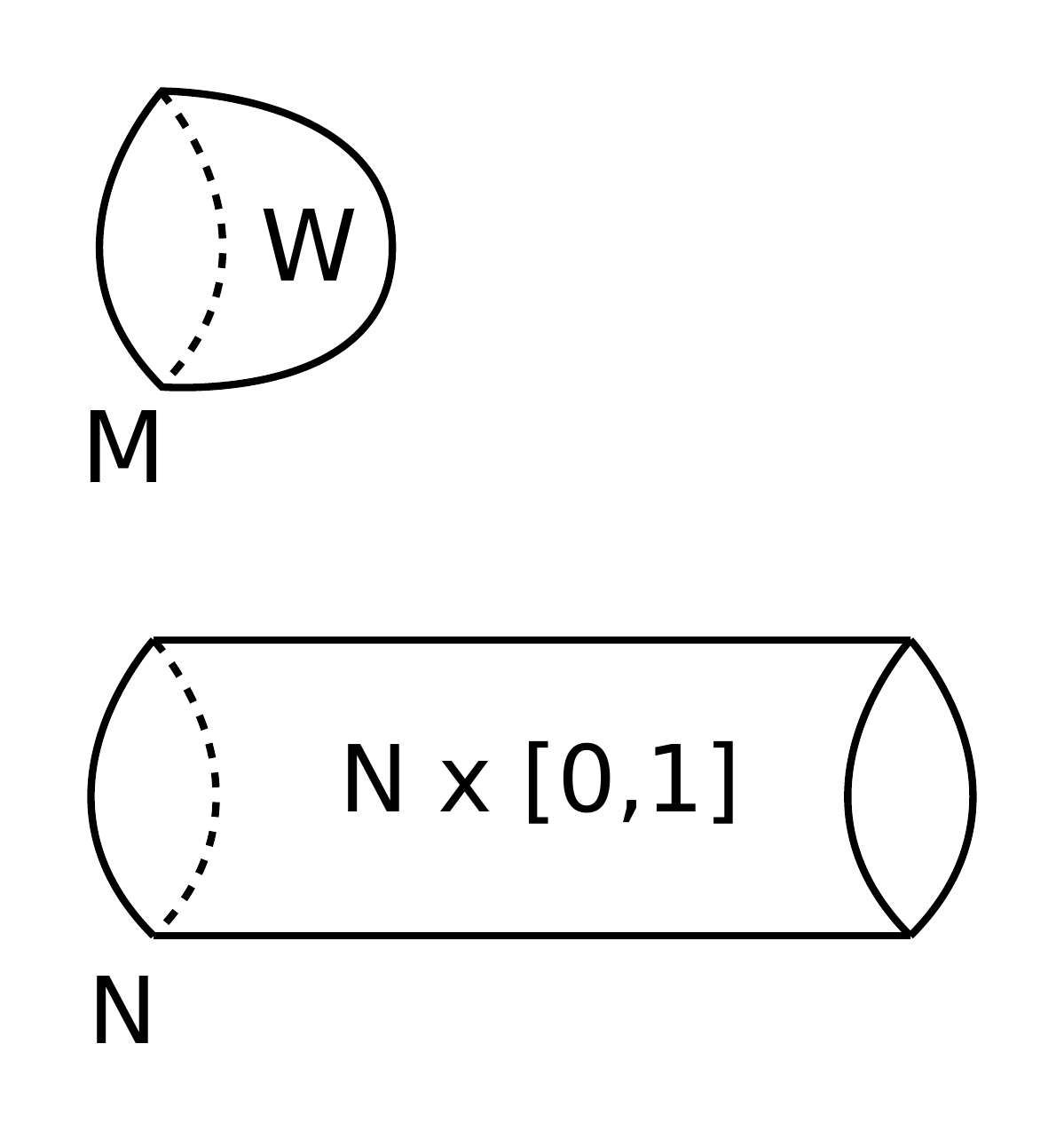}
	\vspace{-1em}
	\caption{The cobordism $Y_W\coloneqq W^\op\amalg N\times[0,1]\colon M\amalg N \leadsto N$.}\label{fig:action}
	\end{figure}
	\begin{figure}
		\centering
		\includegraphics[width=0.95\textwidth]{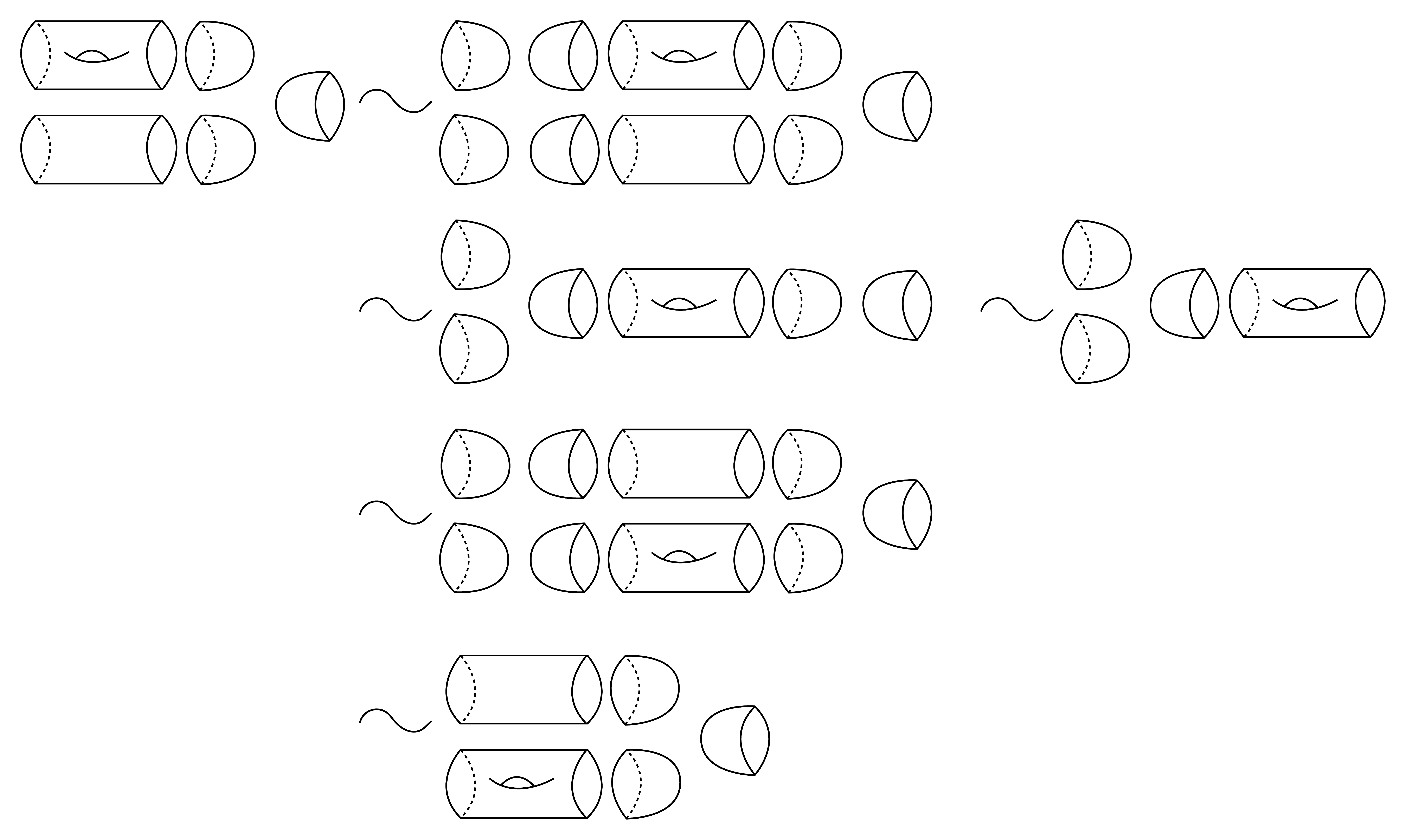}
		\caption{$\mu_W\circ (\calS(U),\id) =\calS(U)\circ\mu_W= \mu_W\circ(\id,\calS(U)).$}\label{fig:units}
	\end{figure}
	
	\begin{prop}\label{prop:action}
	$\rho_W$ defines an action of $\calR^+(M)$ on $\calR^+(N)$ in the homotopy category of spaces.
\end{prop}

\begin{proof}
	We need to show that $\rho_W(e_W,\_)=\id$ and $\rho_W\circ(\id, \rho_W) = \rho_W\circ(\mu_W,\id)$ which follows from \pref{Figure}{fig:action-unital} and \pref{Figure}{fig:action-computation}.
	\begin{figure}[ht]
	\centering
	\includegraphics[width=0.6\textwidth]{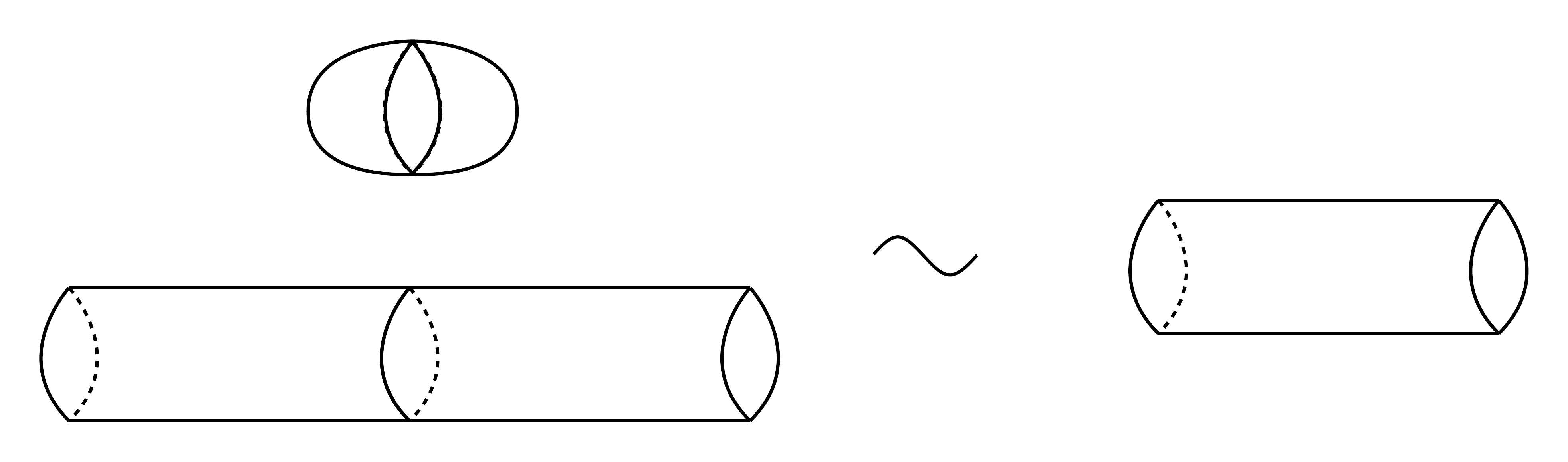}
	\caption{$\rho_W(e_W,\_)=\id$.}\label{fig:action-unital}
	\end{figure}
	\begin{figure}[ht]
	\centering
	\includegraphics[width=0.8\textwidth]{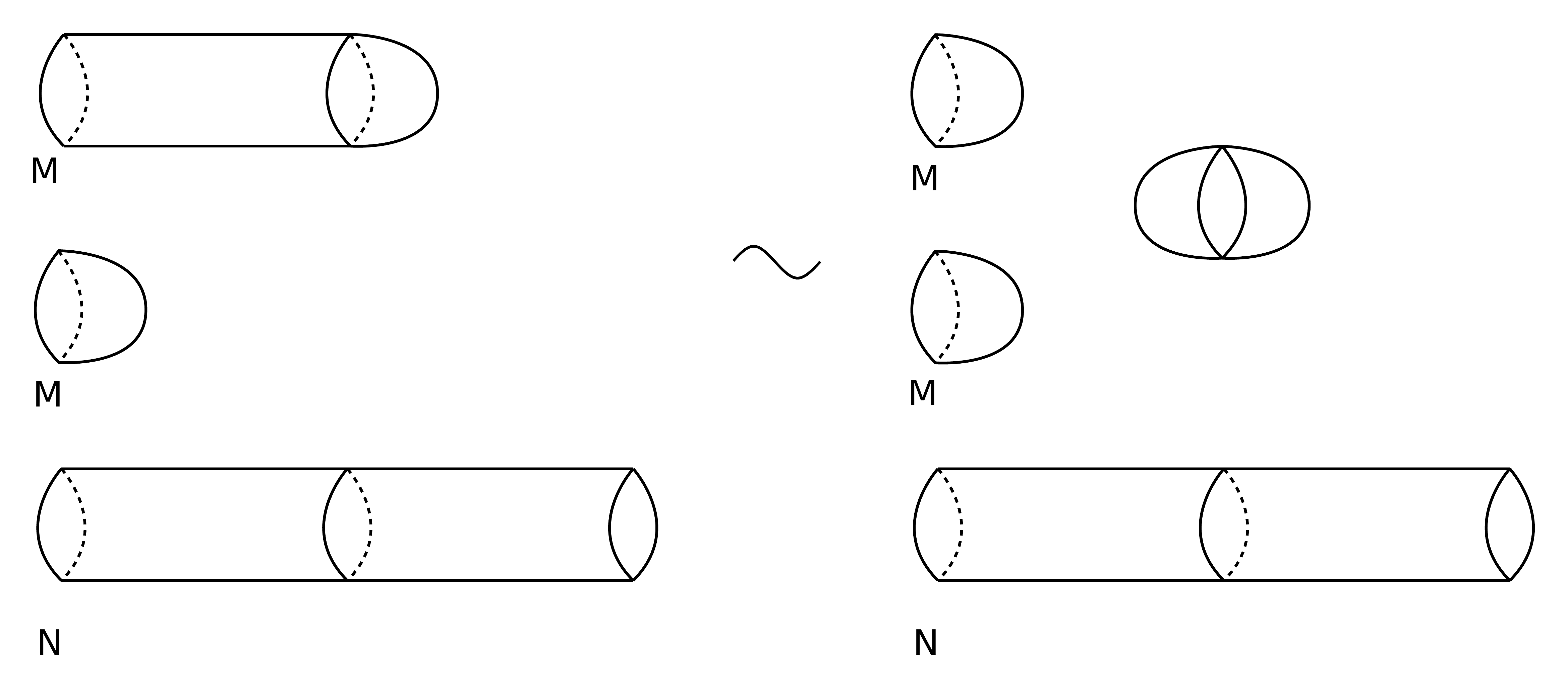}
	\caption{$\rho_W\circ(\id, \rho_W) = \rho_W\circ(\mu_W,\id)$.}\label{fig:action-computation}
	\end{figure}
\end{proof}

\noindent For the final result of this section recall that for a $\theta$-diffeomorphism\footnote{If $\theta\colon B\Spin(d+1)\to B\ort(d+1)$, a $\theta$-diffeomorphism is an orientation preserving diffeomorphism $f\colon N\congarrow N$ together with an isomorphism of $\Spin$-structures $f^*\hat l_N\cong \hat l_N$. For more on general $\theta$-diffeomorphisms see \cite[Section 3.3]{actionofmcg} or \cite[Section 1.2]{ownthesis}.} $f\colon (N,\hat l_N)\to(N,\hat l_N)$ the main result \cite[Theorem A resp. Corollary 3.32]{actionofmcg}\footnote{see also \cite[Corollary B]{ownthesis}} states that the pullback map $f^*\colon \calR^+(N)\to\calR^+(N)$ is homotopic to $\calS(N\times[0,1]\amalg T_f)$, where $T_f$ denotes the \emph{$\theta$-structured mapping torus}.

\begin{thm}\label{thm:diffeo-action}
	Let $f\colon N\to N$ be a $\theta$-diffeomorphism of $N$. If $\calS(M\times[0,1]\amalg T_f)(e_W)$ and $e_W$ lie in the same component of $\calR^+(M)$, then $f^*\colon\calR^+(N)\to\calR^+(N)$ is homotopic to the identity. If furthermore $N$ is $\theta$-nullcobordant, equivalence holds.
\end{thm} 

\begin{remark}
	In particular this shows the following for $N=M$: If $f^*e_V$ and $e_V$ lie in the same path component of $\calR^+(N)$, then $f^*$ is homotopic to the identity.
\end{remark}

\begin{proof}[Proof of \pref{Theorem}{thm:diffeo-action}] 
	The first part is implied by 
	\begin{align*}
		f^* &= \calS(N\times[0,1]\amalg T_f) = \rho_W(e_W,\calS(N\times[0,1]\amalg T_f)) \\
		&= \rho_W(\calS(M\times[0,1]\amalg T_f)(e_W),\id)
	\end{align*}
	where the last equality follows from \pref{Figure}{fig:diffeo-action}. 
	\begin{figure}[ht]
	\centering
	\includegraphics[width=0.8\textwidth]{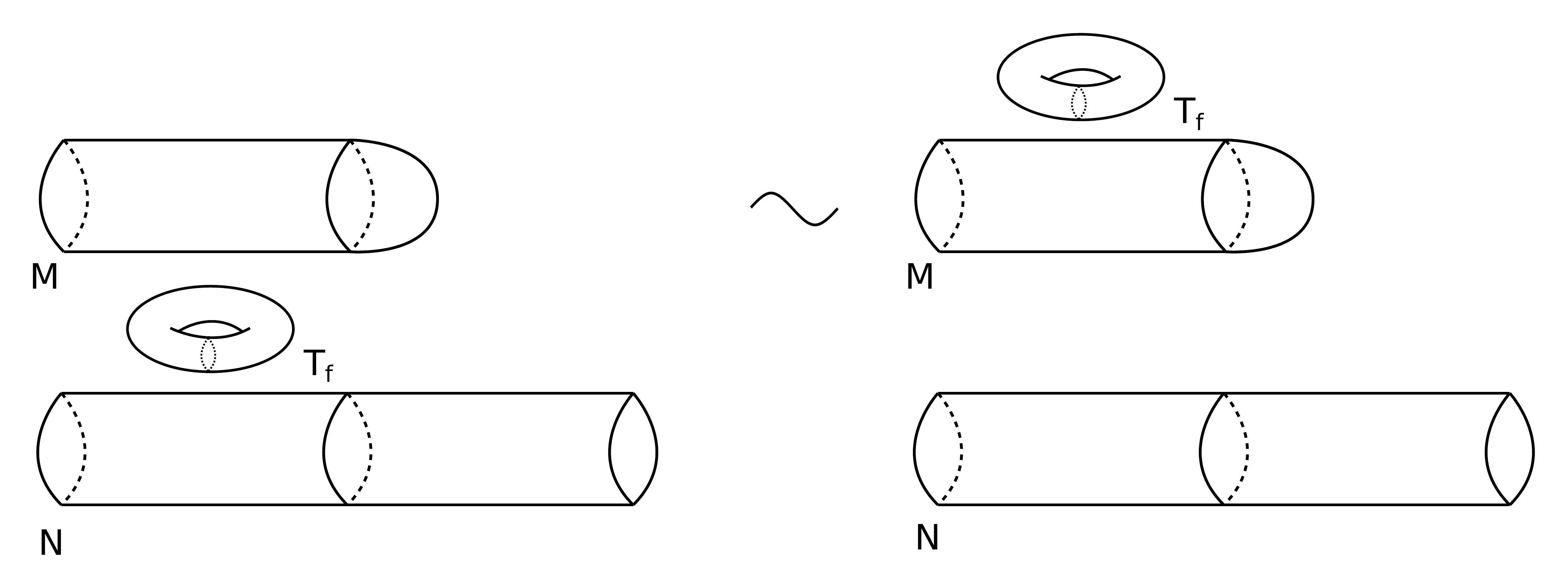}
	\caption{$\rho_W(\id,\calS(N\times[0,1]\amalg T_f)) = \rho_W(\calS(M\times[0,1]\amalg T_f),\id)$.}\label{fig:diffeo-action}
	\end{figure}\\
	If $N$ is $\theta$-nullcobordant as well, say via $V\colon\emptyset\leadsto N$, then $\rho_W = \mu_V(\calS(W^\op\amalg V),\id)$ (see \pref{Figure}{fig:action-nullbordant}) and we compute
	\[\rho_W(\calS(M\times[0,1]\amalg T_f)(e_W),\id) = \mu_V\Bigl(\calS(W^\op\amalg V)\bigl(\calS(M\times[0,1]\amalg T_f)(e_W)\bigr),\id\Bigr).\]
	\begin{figure}[ht]
	\centering
	\includegraphics[width=0.7\textwidth]{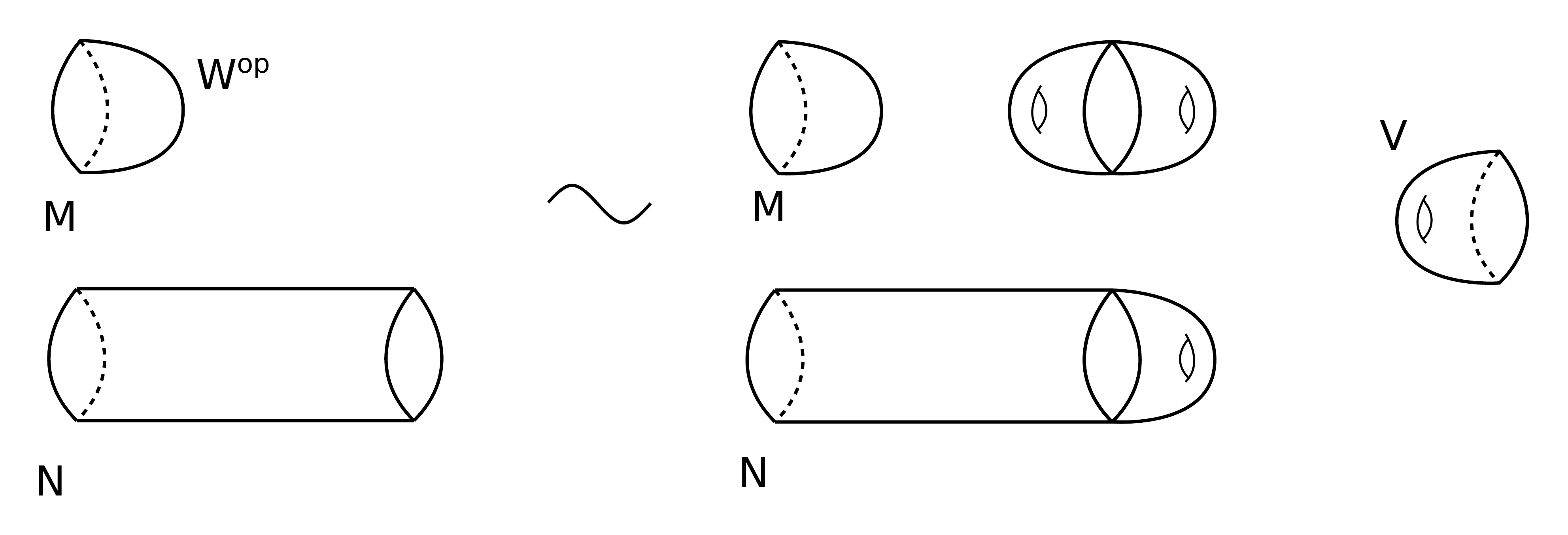}
	\caption{$\rho_W = \mu_V(\calS(W^\op\amalg V),\id)$.}\label{fig:action-nullbordant}
	\end{figure}\\
	This is homotopic to the identity if and only if $\calS(M\times[0,1]\amalg T_f)(e_W)\sim e_W$ since $\calS(W^\op\amalg V)$ is an equivalence of $H$-spaces.
\end{proof}

\noindent Since every orientation preserving diffeomorphism of a simply connected $\Spin$-manifold $N$ lifts to a $\Spin$-diffeomorphism, \pref{Theorem}{thm:diffeo-action-spin} follows immediately from \pref{Theorem}{thm:diffeo-action}.

\noindent As a corollary of the the computation in \pref{Figure}{fig:action-nullbordant} we get:

\begin{cor}
	If $N$ is $\theta$-nullcobordant, then the action of $\calR^+(M)$ on $\calR^+(N)$ is free in the sense that $\rho_W(g) = \id_{\calR^+(N)}$ if and only if $g\sim e_W$.
\end{cor}

\begin{proof}
	Let again $V$ be a $\theta$-nullcobordism of $N$. By \pref{Figure}{fig:action-nullbordant} we have $\rho_W(g) = \mu_V(\calS(W^\op\amalg V)(g),\id)$. It follows that 
	\[\rho_W(g) =\id \quad\iff \quad \calS(W^\op\amalg V)(g)\sim e_V \quad\iff \quad g = e_W,\]
	where the last equivalence holds because $\calS(W^\op\amalg V)$ is an equivalence of $H$-spaces.
\end{proof}

\begin{remark}
The computations from this section rely on the maps being given geometrically via cobordisms. This is reminiscent of quantum field theories which are functors from cobordism categories. Note however, that we also make frequent use of the cobordism relation which allows us to introduce and cancel doubles.
\end{remark}

\section{Comparison to other multiplicative structures on \texorpdfstring{$\calR^+(M)$}{psc-metrics}}\label{sec:comparison}

\subsection{Walsh's construction}
	Let us start by recalling the construction from \cite{walsh_hspaces}. Let $d-1\ge4$ and let $\varphi_i\colon D^{d-1}\embeds S^{d-1}$ be disjoint embeddings for $i=1,2,3$. Let $g_\tor$ be a \emph{torpedo metric} on $D^{d-1}$, i.e. a rotationally symmetric metric of positive scalar curvature that restricts to the cylinder over the round metric near the boundary (cf. \cite[Definition 2.9]{ebertfrenck} for a more precise definition). By the parametrized version of the Gromov--Lawson--Schoen--Yau surgery theorem (cf. \cite{chernysh}, see also \cite{ebertfrenck}) there exists a metric $u'$ on $S^{d-1}\setminus (\im\varphi_1\amalg \varphi_2\amalg\varphi_3)$ such that $u \coloneqq u'\cup(\varphi_1)_*g_\tor\cup(\varphi_2)_*g_\tor\cup(\varphi_3)_*g_\tor\in\calR^+(S^{d-1}, \varphi_1\amalg\varphi_2\amalg\varphi_3)$\footnote{For an embedding $\varphi\colon D^{d-1}\to S^{d-1}$ the space $\calR^+(S^{d-1}, \varphi)$ is defined to be the subspace of those metrics, which have restrict to $\varphi_*g_\tor$ on the image of $\varphi$. If there are several disjoint such embeddings $\varphi_1\amalg\dots\amalg \varphi_n$ the analogous space is denoted by $\calR^+(S^{d-1}, \varphi_1\amalg\dots\amalg\varphi_n)$} lies in the component of the round metric in $\calR^+(S^{d-1})$.\ For clarity let us from now on index the spheres: $S^{d-1}_0$ and $S^{d-1}_1$ will denote the spheres on which we multiply and $S^{d-1}_2$ is the remaining \enquote{reference} sphere. A multiplication map 
	\[\mu^\tor\colon \calR^+(S_0^{d-1},\varphi_1)\times \calR^+(S_1^{d-1},\varphi_1)  \to \calR^+(S^{d-1},\varphi_1)\]
	is given as follows: For $i=0,1$, let $g_i\in\calR^+(S_i^{d-1},\varphi_1)$, say $g_i = g_i'\cup(\varphi_1)_*g_\tor$. We define $\mu^\tor(g_0,g_1)\coloneqq f^*\bigr(g_0'\cup u'\cup g_1\cup (\varphi_1)_*(g_\tor)\bigr)$ for a fixed diffeomorphism 
	\[f\colon S^{d-1}\congarrow \Bigl((S_0^{d-1}\setminus\im\varphi_1)\amalg \bigl(S_2^{d-1}\setminus(\im\varphi_2\cup\im\varphi_3)\bigr)\amalg (S_1^{d-1}\setminus\im\varphi_1)\Bigr)/\sim\]
	The relation identifies $\partial(\im\varphi_1)$ from $S^{d-1}_0$ with $\partial(\im\varphi_2)$ from $S^{d-1}_2$ and $\partial(\im\varphi_1)$ from $S^{d-1}_1$ with $\partial(\im\varphi_3)$ from $S^{d-1}_2$ (see \pref{Figure}{fig:tor-multiplication}). Furthermore we may choose $f$ so that $f\circ\varphi_1 = \varphi_1$ and $\varphi_1$ for $\varphi_1\colon D^{d-1}\embeds S^{d-1}_2$.
	\begin{figure}[ht]
		\includegraphics[width=0.7\textwidth]{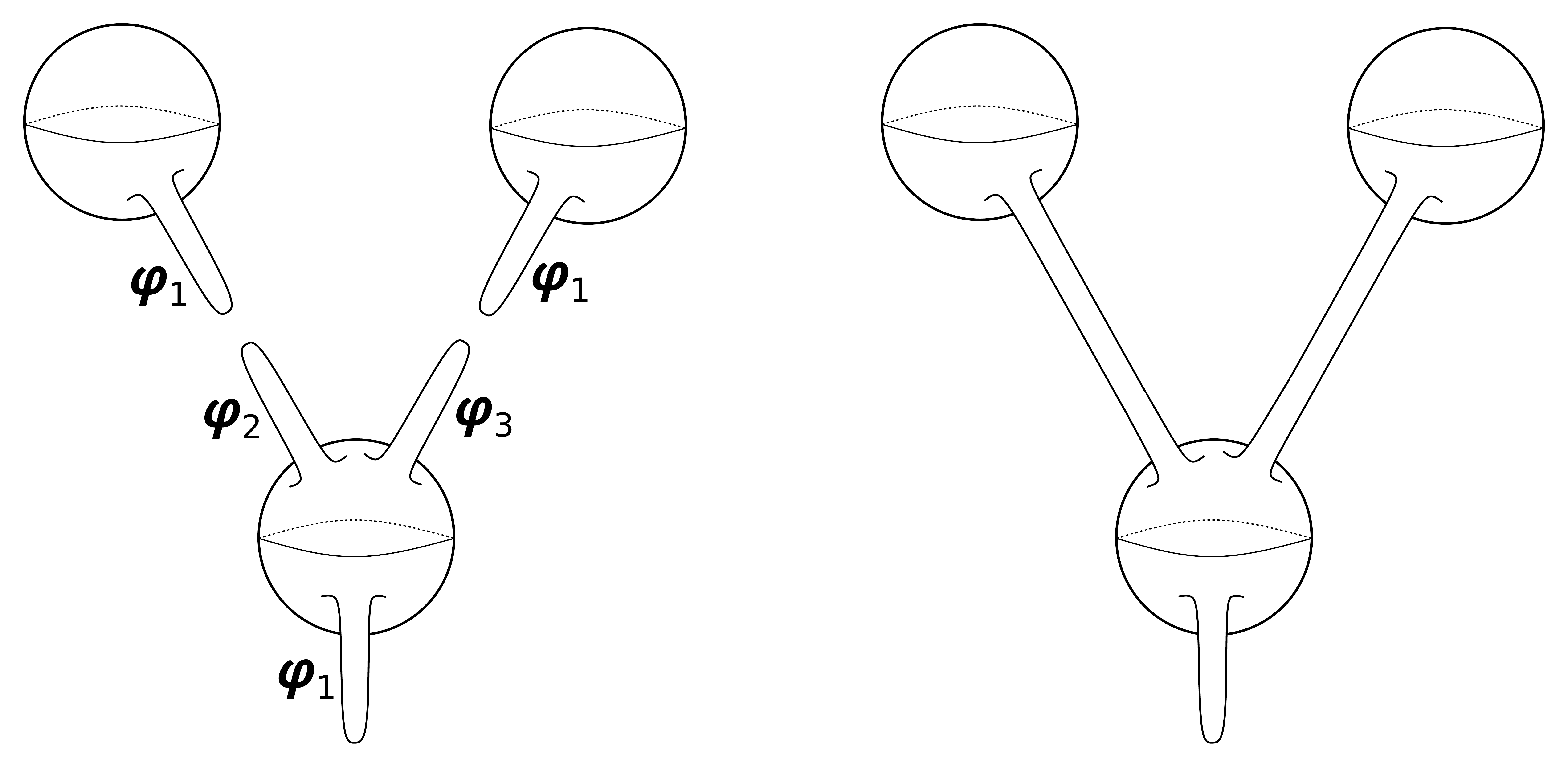}
		\caption{The multiplication $\mu^\tor$.}\label{fig:tor-multiplication}
	\end{figure}
	
	The obtained metric $\mu^\tor(g_0,g_1)$ restricts to $(\varphi_1)_*g_\tor$ on $\im\varphi$ and hence lies in $\calR^+(S^{d-1},\varphi_1)$. Since the inclusion $\calR^+(S^{d-1},\varphi_1)\embeds\calR^+(S^{d-1})$ is a weak equivalence, this defines an $H$-space multiplication $\mu_\tor$ with neutral element given by the round metric on $\calR^+(S^{d-1})$ (cf. \cite[Theorem 5.1]{walsh_hspaces}). It turns out that the component of the round metric $g_\circ$ on $S^{d-1}$ is a $(d-1)$-fold loop space (cf. \cite[Theorem 9.6]{walsh_hspaces}).

%	\begin{proof}[{Proof of \pref{Corollary}{cor:loopspace}}]
%		Let $D\subset M$ be an embedded disk. By \cite[Theorem D]{erw_psc2}, there exists a metric $h\in\calR^+(S^{d-2})$ and a metric $g_\st\in\calR^+(M\setminus D)_h$ such that the inclusion $\calR^+(D)_h\embeds\calR^+(M)$ gluing in $g_\st$ is a homotopy equivalence. The proof follows from the observation that the above multiplication gives an $H$-space structure on $\calR^+(D)_h$ for any boundary condition $h$ such that $\calR^+(D^{d-1})_h\not=\emptyset$. 
%	\end{proof}

	Now let $\varphi_{12}\colon S^0\times D^{d-1}\embeds S^{d-1}_0\amalg S^{d-1}_2$ and $\varphi_{13}\colon S^0\times D^{d-1}\embeds S^{d-1}_1\amalg S^{d-1}_2$ be the disjoint union of $\varphi_1$ with $\varphi_2$ or $\varphi_3$ respectively. Since $u= g_\circ = e_D$, the map $\mu^\tor$ is given by the surgery map for the cobordism (see \pref{Figure}{fig:torpedo-cobordism} for a visualization)
	\[W = \Bigl((S^{d-1}_0\amalg S^{d-1}_1)\times[0,1]\amalg D^d\Bigr)\cup\Bigl(\tr(\varphi_{12})\amalg S^{d-1}_1\times[0,1]\Bigr)\cup\tr(\varphi_{13})\]
	where $\tr$ denotes the trace of a surgery.
	\begin{figure}[ht]
		\includegraphics[width=0.65\textwidth]{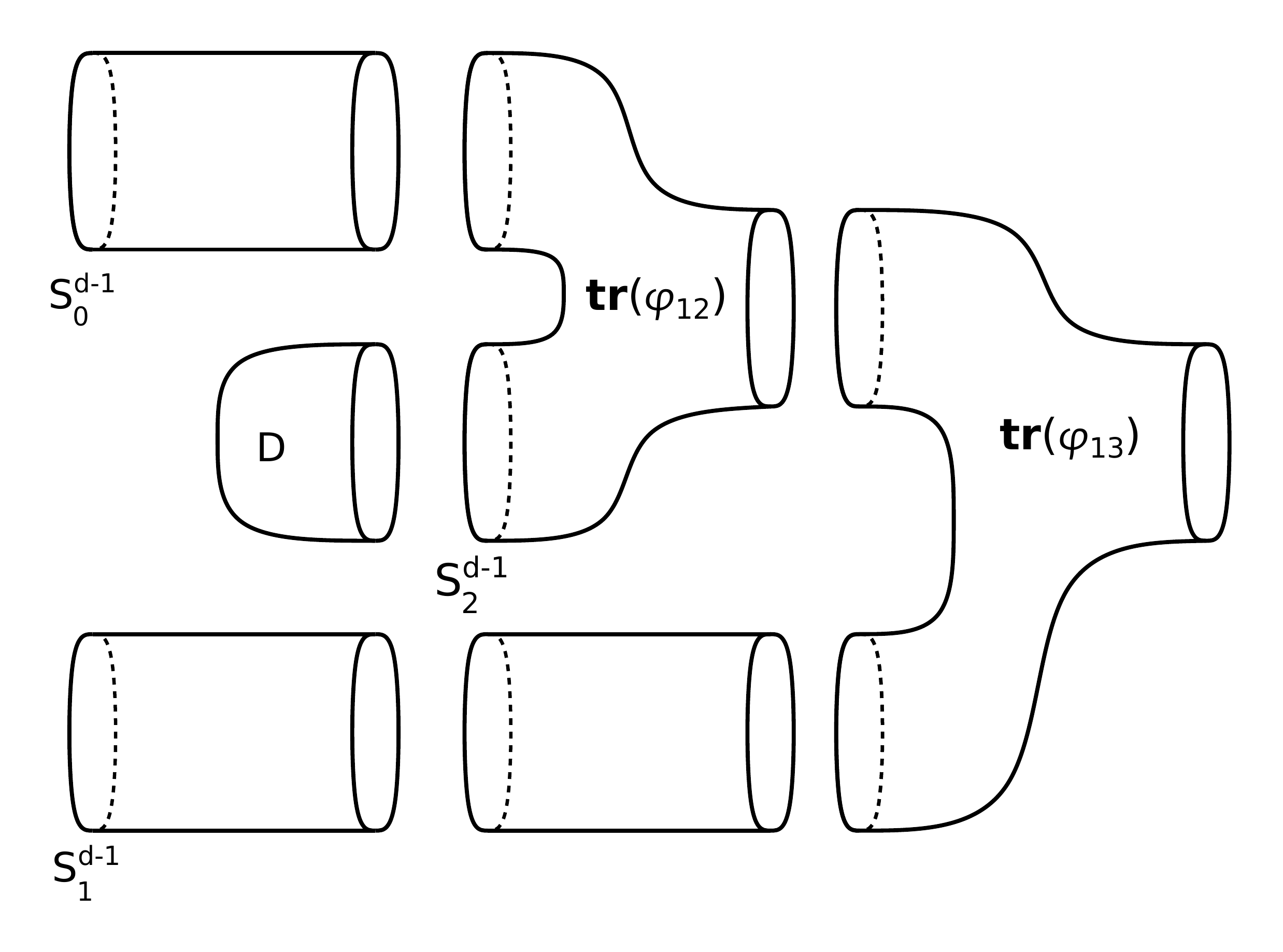}
		\caption{ }\label{fig:torpedo-cobordism}
	\end{figure}\\
	Let $D=D^d\colon \emptyset\leadsto S^{d-1}$ denote the $d$-dimensional disk.
	\begin{prop}
		$W$ is $\Spin$-cobordant to $D^\op\amalg D^\op\amalg D$.
	\end{prop}
	\begin{proof}
		The respective cobordisms $\tr(\varphi_{12})$ and $\tr(\varphi_{13})$ are both $\Spin$-cobordant to $D^\op\amalg D^\op \amalg D$ via connected sum on the interior. So $W$ is cobordant to $D^\op\amalg D^\op\amalg D\amalg 2(D\cup D^\op)$.
	\end{proof}
	\begin{cor}\label{cor:walsh}
		If $d\ge7$, then $\mu^\tor$ and $\mu_{D}$ are homotopic. 
	\end{cor}

\subsection{Stolz's construction}
	Let $M$ be a manifold of dimension $d-1\ge5$ of positive scalar curvature. In \cite{stolz_concordance} Stolz proved the existence of a group structure on concordance classes of psc-metrics on $M$ which was further analysed by Weinberger--Yu in \cite{weinbergeryu} and Xie--Yu--Zeidler \cite{xyz}. For this and the succeeding subsection we need to consider spaces of metrics on manifolds with boundaries. Let $W$ be a manifold with boundary $M$ and let $\calR^+(W)$ denote the space of those psc-metrics on $W$ that restrict to a cylinder $g+dt^2$ in some neighbourhood of the boundary. Since $\scal(g+dt^2) = \scal(g)$, we have a well-defined restriction map
	\[\res\colon\calR^+(W)\too\calR^+(M)\]
	and for $g\in\calR^+(M)$ we define the space $\calR^+(W)_g\coloneqq \res^{-1}(g)$ to consist of those metrics that restrict to $g$ on the boundary. In this situation, we will sometimes call $g$ a \emph{boundary condition}.
	
	\begin{definition}
		Two metrics $g_0, g_1\in\calR^+(M)$ are called \emph{concordant} if there exists a metric $G\in\calR^+(M\times[0,1])_{g_0\amalg g_1}$. The metric $G$ is called a \emph{concordance}. Being concordant is an equivalence relation and we denote the \emph{set of concordance classes of psc-metrics on $M$} by $\ccc(\calR^+(M))$.
	\end{definition}	
	\noindent As a convention we denote concordance classes of metrics by $[g]_c$ and isotopy classes by $[g]$. Since isotopy implies concordance, we get a canonical map $\pi_0(\calR^+(M))\twoheadrightarrow\ccc(\calR^+(M))$. We have the following result:
	\begin{prop}[{\cite[Proposition 3.16 and Remark 3.17]{actionofmcg}}]\label{prop:induced-map}
		Let $\theta$ be the tangential $2$-type of $M_1$ and let $W\colon M_0\leadsto M_1$ be a $\theta$-cobordism. Then $\calS(W)$ induces a map $\ccc(\calR^+(M_0))\to\ccc(\calR^+(M_1))$. Furthermore, if there exists a $G\in\calR^+(W)_{g,h}$, then $\calS(W)([g]_c) = [h]_c$.
	\end{prop}
	\begin{proof}
	Let $G\in\calR^+(W)_{g,h}$ and $\calS(W)([g]_c) = [h']_c$. By \cite[Theorem 3.1]{walsh_parametrized1} there exists $G'\in\calR^+(W)_{g,h'}$ and hence $G^\op\cup G' \in\calR^+(W^\op\cup W)_{h,h'}$ where $G^\op\in\calR^+(W^\op)_{h,g}$ denotes the flipped metric. Now $W^\op\cup W$ is $\theta$-cobordant to $M_1\times[0,1]$ relative to the boundary and by the surgery theorem, there exists a metric $H\in\calR^+(M_1\times[0,1])_{h,h'}$, hence $[h']_c=[h]_c$. The rest has been proven in \cite[Proposition 3.16]{actionofmcg}.
	\end{proof}
	
	\noindent The multiplication of Stolz on $\ccc\calR^+(M)$ is defined as follows. We take the disjoint union of two cylinders over $M$ and consider them as a $\theta$-cobordism from $M\amalg -M\amalg M\leadsto M$ as in \pref{Figure}{fig:conc-multiplication}. Here $-M$ denotes the same underlying manifold with the opposite $\theta$-structure.
	\begin{figure}[ht]
		\includegraphics[width=0.8\textwidth]{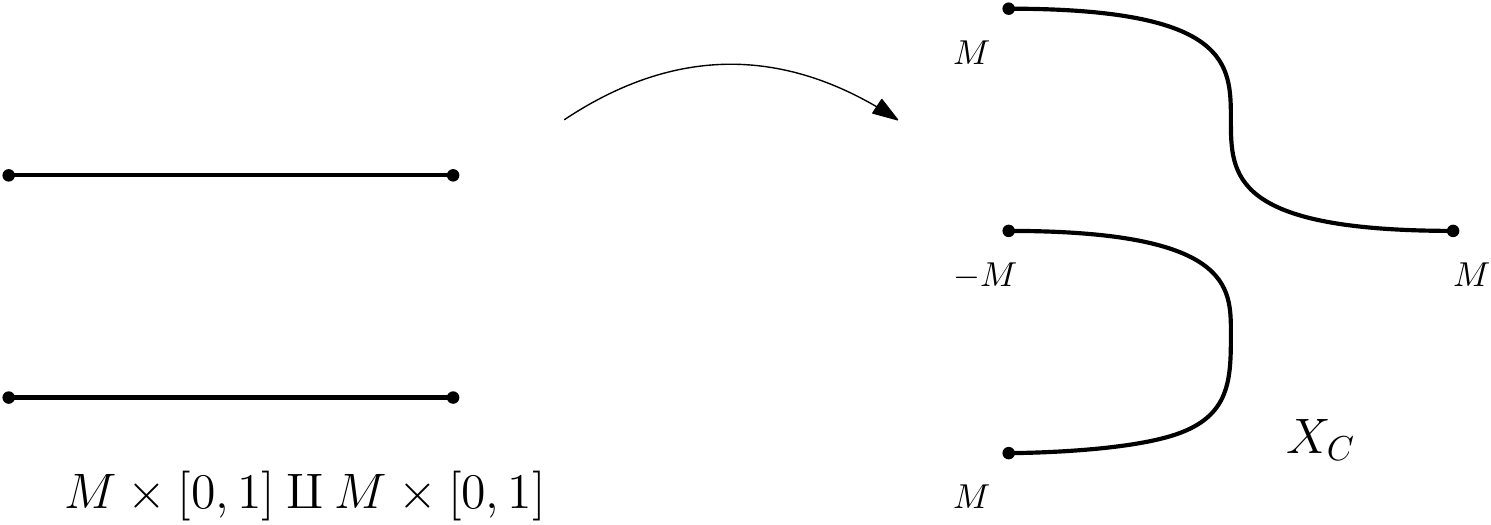}
		\caption{}\label{fig:conc-multiplication}
	\end{figure}
	
	\noindent After performing surgery on this we obtain a cobordism $X_C\colon M\amalg -M\amalg M\leadsto M$ such that the inclusion of the outgoing boundary $M\embeds X_C$ is $2$-connected. Let $u\in\calR^+(M)$ be fixed. The multiplication $\mu^{\conc,u}$ of Stolz is then defined by $\mu^{\conc,u}([g_0]_c,[g_1]_c) = [g]_c$ if there exists a psc-metric $G$ on $X_C$ restricting to $(g_0\amalg u\amalg g_1)\amalg g$ on the boundary. We have the following result relating this multiplication to the surgery map and the $H$-space structure from \pref{Theorem}{thm:main}.
	
	\begin{prop}\leavevmode
		\begin{enumerate}
			\item The map $\mu^{\conc,u}$ is associative, commutative and induced by a map $\calR^+(M)\times\calR^+(M)\to\calR^+(M)$ of spaces.
			\item If $M$ is nullcobordant in its own tangential $2$-type via a nullcobordism $W\colon \emptyset\leadsto M$, then $\mu^{\conc,e_W} = \mu_W$.
		\end{enumerate}
	\end{prop}
	
	\begin{proof}\leavevmode
	\begin{enumerate}
		\item It follows directly that from \pref{Proposition}{prop:induced-map} that $\mu^{{\conc,u}}(g_0,g_1) = [\calS_{X_C}(g_0,u,g_1)]_c$ and so the multiplication $\mu^{\conc,u}$ is induced by the map $\calS_{X_C}$. Associativity and commutativity of $\mu^{\conc,u}$ can then be proven using graphical calculus, where we mark the part incoming boundary that does not belong to the multiplication by $u$ (see \pref{Figure}{fig:conc_comm} and \pref{Figure}{fig:conc_assoc}).
	\begin{figure}[ht]
		\includegraphics[width=0.9\textwidth]{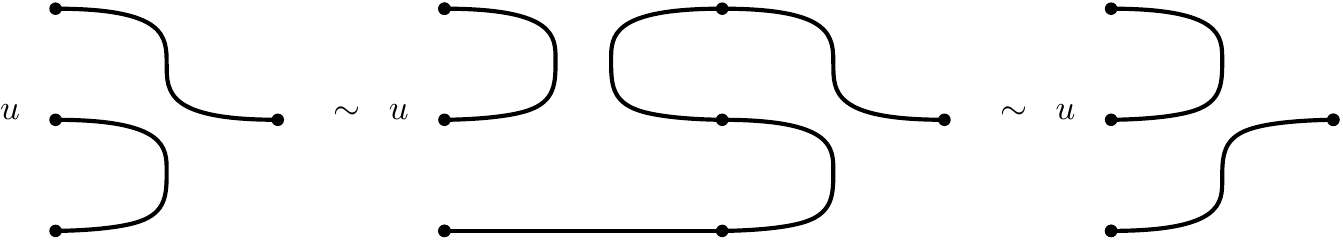}
		\caption{Commutativity of $\mu^{\conc,u}$}\label{fig:conc_comm}
	\end{figure}	
	\begin{figure}[ht]
		\includegraphics[width=0.8\textwidth]{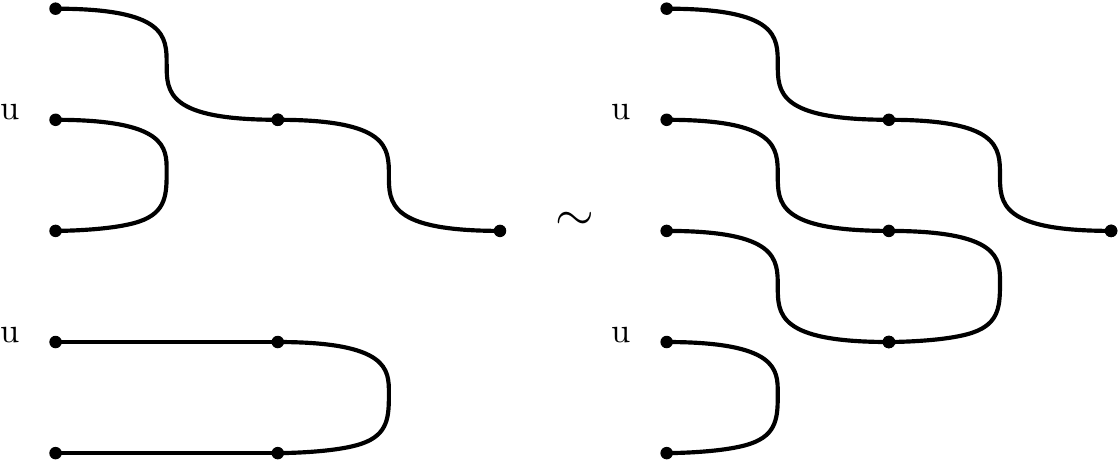}
		\caption{Associativity of $\mu^{\conc,u}$}\label{fig:conc_assoc}
	\end{figure}

		\item Let $M$ is nullcobordant in its own tangential $2$-type via a nullcobordism $W\colon \emptyset\leadsto M$. Since $X_W\sim (M\times[0,1]\amalg  -W\amalg M\times[0,1])\cup X_C$ (see \pref{Figure}{fig:capping-off}), we have:
	\[\mu^{\conc,e_W} = \calS_{X_C}(\_,e_W,\_)=\mu_W.\]\qedhere
	\begin{figure}[ht]
		\includegraphics[width=0.8\textwidth]{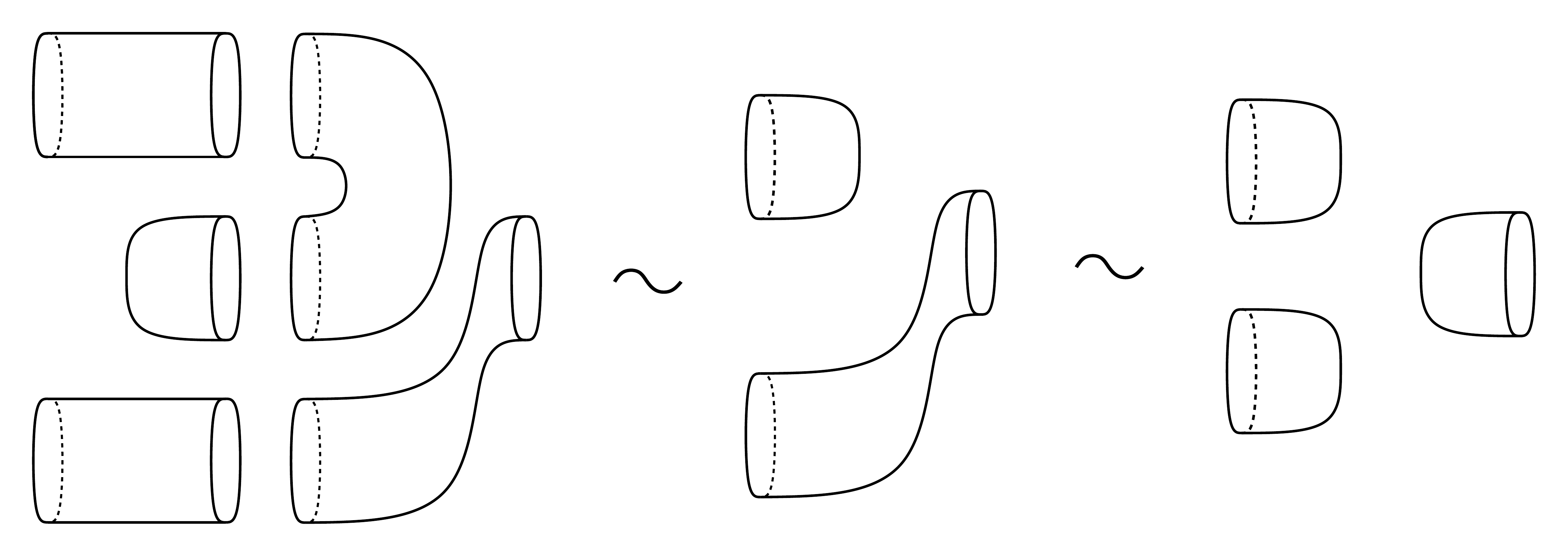}
		\caption{$\calS_{X_C}(\_,e_W,\_)=\mu_W.$}\label{fig:capping-off}
	\end{figure}
	\end{enumerate}
	\end{proof}

\subsection{Concatenation of cylinders over spheres}

Let $M^{d-2}$ be a manifold and $g\in\calR^+(M)$. Then $e_{\cyl}\coloneqq g+dt^2\in \calR^+(M\times[0,1])_{g,g}$ and $\calR^+(M\times[0,1])_{g\amalg g}$ becomes a homotopy-associative $H$-space with the multiplication map $\mu_\cyl$ given by $(G,G')\mapsto G\cup G'$ and appropriately rescaling back to $M\times[0,1]$. The neutral element is given by $e_\cyl$. It has been shown in \cite[Theorem B]{erw_psc3}, that the components of invertible elements of $\calR^+(M\times[0,1])_{g\amalg g}$ carry an infinite loop space structure with underlying $H$-space multiplication given by $\mu_\cyl$, provided $M$ admits a relatively $2$-connected nullcobordism. Before further studying this multiplication, we need to recall the notion of (right-)stable metrics due to Ebert--Randal-Williams.

\begin{definition}[{\cite[Definition 3.1.1]{erw_psc2}}]
	Let $W\colon M_0\leadsto M_1$ be a cobordism and let $g_i\in\calR^+(M_i)$ be boundary conditions. A psc-metric $G\in \calR^+(W)_{g_0\amalg g_1}$ is called \emph{right-stable}, if for every cobordism $W'\colon M_1\leadsto M_2$ and every boundary condition $g_2\in\calR^+(M_2)$, the map
	\begin{align*}
		\calR^+(W')_{g_1\amalg g_2}&\to\calR^+(W\cup W')_{g_0\amalg g_2}\\
			H&\mapsto G\cup H
	\end{align*}
	is a weak equivalence. Similarly, $G$ is called left-stable if the map $H'\mapsto H'\cup G$ is a weak equivalence for every cobordism $W''\colon M_{-1}\leadsto M_0$.
\end{definition}

\noindent Now, let the manifold $M$ from above be nullcobordant via $N\colon\emptyset\leadsto M$ such that the pair $(N,M)$ is $2$-connected. Then, by \cite[Theorem D]{erw_psc2} there exists a boundary condition $g\in\calR^+(M)$ such that $N$ admits a right-stable metric $G_\rst\in\calR^+(N)_g$, provided that $d\ge7$.  We note that the metric $G_\rst^\op\in\calR^+(N^\op)_g$ obtained by flipping $G_\rst$ is left-stable and therefore we have a homotopy equivalence
\[\cl_{G_\rst}\colon \calR^+(M\times[0,1])_{g,g} \too\calR^+(N\cup M\times[0,1]\cup N^\op)\]
defined by mapping a metric $G$ on $M\times[0,1]$ to $G_\rst\cup G\cup G_\rst^\op$, i.e. it is given by gluing in $G_\rst$ on both $N$ and $N^\op$. Note that $N\cup M\times[0,1]\cup N^\op = dN$ is diffeomorphic to the double of $N$. Since doubles are nullcobordant (cf. \pref{Proposition}{prop:double-is-nullbordant}), there exists a nullcobordism $W\colon \emptyset\leadsto dN$.
\begin{samepage}
\begin{que}\label{que:cylinder}\leavevmode
	\begin{enumerate}
		\item Is there a boundary condition $g\in\calR^+(M)$ and a $\theta$-nullcobordism $W\colon \emptyset\leadsto dN$ such that there exists an equivalence of $H$-spaces 
		\[ (\calR^+(M\times[0,1])_{g\amalg g},\mu_\cyl) \to (\calR^+(dN),\mu_W)?\] 
		\item If so, can one choose $W$ and $g$ such that there exists a right-stable metric $G_\rst\in\calR^+(N)_g$ for which the map $\cl_{G_\rst}$ is an equivalence? 
	\end{enumerate}
\end{que}
\end{samepage}

\noindent The natural starting point for investigating this question is the case that $M=S^{d-2}$, $g=g_\circ^{d-2}$ is the round metric, $N=D^{d-1}$, $G_\rst=g_\tor$ is the torpedo metric and $W=D\coloneqq D^d$. We identify $dD^{d-1}=S^{d-1}=\partial W$. In this case it is possible to get a more explicit form of the multiplication map $\mu_W$: Let $\varphi\colon S^0\times D^{d-1}\embeds S^{d-1}\amalg S^{d-1}$ be the inclusion of the lower hemisphere into the first and the upper hemisphere into the second factor. We define the map $\overline\calS_\varphi\colon \calR^+(S^{d-1}\amalg S^{d-1}, \varphi)\to \calR^+(S^{d-1})$ by
\begin{equation}\label{equ:1}
\overline\calS_\varphi((g\cup g_\tor^\op)\amalg (g_\tor\cup g')) = g\cup (g^{d-2}_\circ +dt^2) \cup g'.
\end{equation}
By the parametrized version of the Gromov--Lawson--Schoen--Yau surgery theorem (\cite{chernysh}, see also \cite{ebertfrenck}) the inclusion map $\calR^+(S^{d-1}\amalg S^{d-1}, \varphi)\embeds\calR^+(S^{d-1}\amalg S^{d-1})$ is a weak homotopy equivalence and we denote the composition of its homotopy inverse with $\overline\calS_\varphi$ by $\calS_\varphi$. By definition (see \cite[Definition 2.23 (3)]{actionofmcg}\footnote{see also \cite[Definition 3.1.1 (3)]{ownthesis}}) this agrees with $\calS(X_W)$ and the map $\mu_W$ is therefore homotopic to $\calS_\varphi$. Consider the following diagram
\begin{center}
\begin{tikzpicture}
	\node (01) at (0,4.5) {$\calR^+(S^{d-1})\times\calR^+(S^{d-1})$};
	\node (0) at (0,3) {$\calR^+(S^{d-1}\amalg S^{d-1})$};
	\node (1) at (0,1.5) {$\calR^+(S^{d-1}\amalg S^{d-1},\varphi)$};
	\node (2) at (0,0) {$\calR^+(S^{d-2}\times[0,1])_{g_\circ,g_\circ}\times \calR^+(S^{d-2}\times[0,1])_{g_\circ,g_\circ}$};
	\node (3) at (7,3) {$\calR^+(S^{d-1})$};
	\node (4) at (7,0) {$\calR^+(S^{d-2}\times[0,1])_{g_\circ,g_\circ}$};
	
	\draw[->](01) to node[above]{$\mu_W$} (3);
	\draw[double equal sign distance](01) to (0);
	\draw[->](0) to node[above]{$\calS_\varphi$} (3);
	\draw[->](1) to node[above]{$\overline\calS_\varphi$} (3);
	\draw[left hook->](1) to node[right]{$\simeq$} (0);
	\draw[->](2) to node[right]{$\cl_{g_\tor}\times\cl_{g_\tor}$} (1);
	\draw[->](2) to node[above]{$\mu_\cyl$} (4);
	\draw[->](4) to node[right]{$\cl_{g_\tor}$} (3);
\end{tikzpicture}
\end{center}
where the triangles commute up to homotopy by the definition and the cobordism invariance of $\calS$ and the lower square commutes up to homotopy by \pref{Equation}{equ:1} after appropriately rescaling the cylinders. We therefore can affirmatively answer \pref{Question}{que:cylinder} in this special case:
\begin{thm}\label{thm:compare}
	The map $\cl_{g_\tor}\colon(\calR^+(S^{d-2}\times[0,1])_{g_\circ,g_\circ},\mu_\cyl) \to (\calR^+(S^{d-1}),\mu_D)$ is an equivalence of $H$-spaces provided $d\ge7$.
\end{thm}
\noindent \pref{Corollary}{cor:compare} now follows from \pref{Corollary}{cor:walsh} and \pref{Theorem}{thm:compare}.

\let\oldaddcontentsline\addcontentsline %Store \addcontentsline
\renewcommand{\addcontentsline}[3]{}%Make \addcontentsline a no-op
\bibliographystyle{halpha-abbrv}
\bibliography{Bibliography}
\let\addcontentsline \oldaddcontentsline%Restore \addcontentsline

\end{document}